\documentclass[a4paper,11pt]{article}
\usepackage[latin1]{inputenc}
\usepackage[english]{babel}
\usepackage[T1]{fontenc}
\usepackage{amsfonts,amssymb,amsthm,amsmath,amsbsy,eucal}
\usepackage{mathrsfs}
\usepackage{bbm}
\usepackage[all]{xy}
\usepackage{graphicx}
\usepackage[babel]{csquotes}
\usepackage{enumerate}
\usepackage{setspace}
\usepackage{fullpage}

\usepackage{epigraph}
\newtheorem{teor}{Theorem}[section]
\newtheorem*{theorem*}{Theorem}
\newtheorem{lemma}[teor]{Lemma}
\newtheorem{prop}[teor]{Proposition}

\theoremstyle{definition}
\newtheorem{defin}[teor]{Definition}
\theoremstyle{remark}
\newtheorem{rmk}[teor]{Remark}

\newcommand{\OF}{\mathcal{O}_F}
\newcommand{\Z}{\mathbb{Z}}
\newcommand{\N}{\mathbb{N}}
\newcommand{\C}{\mathbb{C}}
\newcommand{\F}{\mathcal{F}}
\newcommand{\Ff}{\mathbb{F}}
\newcommand{\Fou}{\mathscr{F}}
\newcommand{\I}{\mathbbm{1}}
\renewcommand{\S}{\mathcal{S}}
\newcommand{\Rx}{R^{\times}}
\newcommand{\la}{\langle}
\newcommand{\ra}{\rangle}
\newcommand{\B}{\mathrm{B}_0(W)}
\newcommand{\BB}{\mathbb{B}_0(W)}
\newcommand{\A}{\mathbb{A}(W)}
\newcommand{\Mp}{\mathrm{Mp}(W)}
\newcommand{\Mpn}[1]{\mathrm{Mp}_#1(W)}
\newcommand{\Sp}{\mathrm{Sp}(W)}

\renewcommand{\d}{\mathbf{d}}
\renewcommand{\t}{\mathbf{t}}
\renewcommand{\r}{\mathbf{r}}
\newcommand{\s}{\mathbf{s}}
\newcommand{\End}{\mathrm{End}}
\newcommand{\Aut}{\mathrm{Aut}}
\newcommand{\Isom}{\mathrm{Iso}}
\newcommand{\virg}[1]{\textquotedblleft #1\textquotedblright}
\newcommand{\Hom}{\mathrm{Hom}}

\makeindex

\begin{document}
\title{\Large Weil representation and metaplectic groups over an integral domain}
\author{\large Gianmarco Chinello \and Daniele Turchetti}
\date{}
\maketitle


\begin{abstract}

\small Given $F$ a locally compact, non-discrete, non-archimedean field of characteristic $\neq 2$ and $R$ an integral domain such that a non-trivial smooth character $\chi: F\to \Rx$ exists, we construct the (reduced) metaplectic group attached to $\chi$ and $R$. We show that it is in most cases a double cover of the symplectic group over $F$. Finally we define a faithful infinite dimensional $R$-representation of the metaplectic group analogue to the Weil representation in the complex case.
\end{abstract}



\section*{Introduction}

\hspace{3ex} The present article deals with the seminal work of André Weil on the Heisenberg representation and the metaplectic group. In \cite{Wei} the author gives an interpretation of the behavior of theta functions throughout the definition of the metaplectic group with a complex linear representation attached to it, known as the Weil or metaplectic representation. A central tool in his construction is the group $T=\{z\in \C : |z|=1\}$, in which most computations are developed.\\
We replace $T$ with the multiplicative group of an integral domain $R$ and we construct a Weil representation in this more general context. The scope is to help fitting Weil's theory to give applications in the setting of modular representations (see, for example, \cite{Min2}).

\bigskip
The classical results of \cite{Wei} are the following. Let $X$ be a finite dimensional vector space over a local field $k$, $X^*$ its dual and $W=X\times X^*$. Let $A(W)$ be the product $W\times T$ with the Heisenberg group structure as defined in section \ref{defB0} below. The author studies the projective representation of the symplectic group  $\Sp$, coming from the action of $\Sp$ over a complex representation of $A(W)$. This projective representation lifts to an actual representation of a central extension $\Mp$ of $\Sp$, called \textbf{metaplectic group}. The lift is called nowadays \textbf{Weil representation} or metaplectic representation. The author shows also that the metaplectic group contains properly a subgroup $\mathrm{Mp}_2(W)$ which is a two-folded cover of $\Sp$ on which the Weil representation can be restricted. Moreover, if $k\neq\C$, it is not possible to restrict the Weil representation to $\Sp$. 

\bigskip
Let us recall another construction for the Weil representation in the complex case (cfr. for example \cite{MVW}). The starting point is the Stone-von Neumann theorem, asserting that, given a non-trivial character $\chi:k\to \C^\times$, there exists an infinite dimensional irreducible $\C$-representation of the Heisenberg group $$\rho: W\times k \to \Aut(S)$$ with central character $\chi$, and that it is unique up to isomorphism. The symplectic group acts on the Heisenberg group by $\sigma.(w,x)\mapsto (\sigma w, x)$ and this action is trivial on the center. Then, for every $\sigma\in \Sp$, the representation $\rho_\sigma:(w, x)\mapsto \rho(\sigma w, x)$ is irreducible and has the same central character $\chi$, so it is isomorphic to $\rho$. This means that there exists $\Psi_\sigma\in \Aut(S)$ such that $\Psi_\sigma\circ \rho \circ \Psi_\sigma^{-1}= \rho_\sigma. $
Notice that $\Psi_\sigma$ is unique up to multiplication by an element of $\C^\times$, by Schur's Lemma. We obtain in this way a faithful projective representation $$\begin{array}{ccc}
    \Sp & \to & \Aut(S)/\C^\times \\
    \sigma &\mapsto & \Psi_\sigma.\\
\end{array}$$
Defining $$\mathrm{Mp}_\C(W):=\Sp\times_{\Aut(S)/\C^\times}\Aut(S)$$ the metaplectic group comes out, by definition, with a representation that lifts the projective represntation of $\Sp$: the complex Weil representation.\\ Notice that this construction is more abstract than the one in \cite{Wei} and that repose on the irreducibility of $\rho$, which in general is not given when we replace $\C$ by $R$. We want to avoid the use of Stone-Von Neumann theorem and to give an explicit description of the Weil representation. This is why we choose to follow the approach of Weil rather than this construction.

\bigskip
Let $F$ be a locally compact, non-discrete, non-archimedean field of characteristic $\neq 2$. Let $p>0$ be the characteristic and $q$ the cardinality of its residue field.
We let $X$ be a $F$-vector space of finite dimension, we note $W=X\times X^*$ and we replace the group $T$ by the multiplicative group of an 
integral domain $R$ such that $p\in R^\times$, $R$ contains $p^n$-th roots of unity for every $n$ (to ensure the existence of a nontrivial smooth character $F\to R^\times$) and a square root of $q$. The object of this work is to show that the strategy of proof used by Weil can be adapted in this setting. Weil's techniques can be exploited in the same way whenever a result involves just the field $F$, like the intrinsic theory of quadratic forms over $X$ and the description of the symplectic group. Nevertheless, different kinds of problems occur in the new generality. The main issues are the lack of complex conjugation and complex absolute value. Because of this, Fourier and integration theory in the present work are different from the complex case; mainly we consider Haar measures with values in $R$ and operators acting over the space of $R$-valued Schwartz functions over an $F$-vector space instead of $L^2$-functions, using Vignéras' approach  (section I.2 of \cite{Vig1}). Moreover, allowing $R$ to be of positive characteristic makes it necessary to change some formulas, for example in the proof of Theorem \ref{gamma-1} to include the case where $q^2=1$ in $R$.

\bigskip
The central result of this paper is the existence of the \textbf{reduced metaplectic group}, which is defined in the following way. The starting point is the definition of the metaplectic group $\Mp$ and the existence of a non-split short exact sequence
\begin{equation}\tag{$\star$}\label{ses0}
1\longrightarrow \Rx \longrightarrow \Mp {\longrightarrow}\Sp \longrightarrow 1.
\end{equation}
Theorem \ref{rmg} and Theorem \ref{nonsplitta} give a description of a minimal subgroup of $\Mp$ which is a non-trivial extension of $\Sp$. We can summarize them in a unique statement:

\begin{theorem*}
Let $\mathrm{char}(R)\neq 2$. There exists a subgroup $\Mpn{2}$ of $\Mp$ such that the short exact sequence $($\ref{ses0}$)$ restricts to a short exact sequence 
\begin{equation*}\tag{$\star\star$}
1\longrightarrow \{\pm 1\} \longrightarrow \Mpn{2} {\longrightarrow}\Sp \longrightarrow 1
\end{equation*}
that does not split.
\end{theorem*}
\noindent
This result permits the definition of a Weil representation of $\Mp$, that we describe explicitly.

\newpage
\noindent
Let us describe the main body of the article.

\bigskip \noindent
Section $1$ contains a brief explanation of basic notations and definitions where essentially no new result appears. However some features differ from the one in \cite{Wei}. We introduce the integration theory in our setting that is slightly different from the complex one. Over a $F$-vector space, we consider a $R$-valued Haar measure as in \cite{Vig1}, that exists since $p\in \Rx$ and the $R$-module of Schwartz functions, i.e. compactly supported locally constant functions, in place of $L^2$-functions. The main differences with the complex case are that may exist non-empty open subsets of the vector space with zero volume if the characteristic of $R$ is positive, and that integrals of Schwartz functions are actually finite sums. This theory permits also the definition of a Fourier transform and of its inverse. In the end of the section we study element of $\Sp$ as matrices acting over $W$. We consider this as a left action (rather than on the right, as in \cite{Wei}) but we want to show the same formulas. Then we have to change some definitions \textit{ad hoc}.\\
In section $2$ we define the faithful \textbf{Heisenberg representation} $U$ of $A(W)$ on the $R$-module of Schwartz functions of $X$ and we introduce the groups $\B$ of automorphisms of $A(W)$ acting trivially on the center and $\BB$, the normalizer of $U(A(W))$ in $\Aut(\S(X))$. After that we define
$\Mp$, as a fibered product of $\Sp$ and $\BB$ over $\B$, and the sequence (\ref{ses0}), proving that it is exact. This fact is a direct consequence of the Theorem \ref{ses1}, stating the exactness of a sequence of the form
\begin{equation*}
1\longrightarrow R^\times \longrightarrow \BB \overset{\pi_0}{\longrightarrow} \B \longrightarrow 1.
\end{equation*}
The proof of the analogue of Theorem \ref{ses1} in sections 8, 9, 10 of \cite{Wei} uses a construction that has been introduced by Segal in the setting of complex unitary operators for $L^2$-functions (cfr. chapter 2 of \cite{Seg}). It is indeed possible to mimic it for Schwartz functions over $R$, but this does not yield surjectivity of $\pi_0$ when $R$ has not unique factorization. In fact Lemma 2 in \cite{Wei} does not hold in our setting. To get around this problem we give explicit generators of $\B$ and we show that they are in the image of $\pi_0$.\\
In section $3$ we define the Weil factor $\gamma(f)\in R^\times$, associated to a quadratic form $f$ over $F$. It is the constant that permits to transpose some relations between maps taking values in $\B$, to the liftings of those maps, that take values in $\BB$. We prove some properties of the map $\gamma:f\mapsto \gamma(f)$ and an explicit summation formula for $\gamma(f)$.\\ 
In section $4$ we go further into the study of properties of the Weil factor. In Theorem \ref{gamma-1} we prove that $\gamma(n)=-1$, where $n$ is the reduced norm over the quaternion algebra over $F$. To prove this theorem we can not use directly the proof in \cite{Wei} since the author shows that $\gamma(n)$ is a negative real number of absolute value $1$ and computes integrals on subsets that may have volume zero when $R$ is of positive characteristic. The key tool is the summation formula proved in section 3. We also show that $\gamma$ respects the Witt group structure over the set of quadratic forms and, combining this with known results over quadratic forms, we show that $\gamma(f)$ is at most a fourth root of unity in $R^\times$. \\
Finally, in section $5$, we use the results from previous sections to construct the reduced metaplectic group and prove the main theorem: we build up a $R$-character of $\Mp$, whose restriction on $\Rx$ is the map $x\mapsto x^2$, and we define $\Mpn{2}$ as its kernel. We prove that $\Mpn{2}$ is a cover of $\Sp$ with kernel the group of square roots of unity in $R$, so that if $\mathrm{char}(R)=2$ the sequence (\ref{ses0}) splits.

\bigskip
The existence of a Weil representation over $R$ is a result which is strongly motivated by recent research problems. Minguez studies local theta correspondences in \cite{Min1} and $p$-adic reductive pairs in the $\ell$-modular case in \cite{Min2}. Here he asks how does Howe correspondence behave with respect to reduction modulo $\ell$ and he suggests that a Weil representation over $\bar \Ff_\ell$ has to be constructed and the theory of the metaplectic group has to be extended. A general construction with a strong geometric flavor is given by Shin to study the case of representations over $\bar \Ff_p$, which is not possible to treat following a naive approach. In \cite{Shi}, the author defines in great generality $p$-adic Heisenberg group schemes over a noetherian scheme. He proves a geometric analogue of Stone-Von Neumann theorem and Schur's lemma. Thanks to this he is able to construct a Weil representation provided the existence of a (geometric) Heisenberg representation. Showing that the latter exists for every algebraically closed field in every characteristic, he is able to define the new notion of mod $p$-Weil representation. The great advantage of his construction is in fact the possibility to treat the case where $\mathrm{char}(R) = p$ (in this case every character $F\to \Rx$ is trivial, so one does really need another approach). On the other hand an elementary approach, like the one in the present article, permits to define a Weil representation over integral domains that are not fields. The possibility of working in this generality is motivated, among other things, by the recent interest in representation theory of reductive groups over discrete valuation rings. We cite, for example, the paper \cite{EmHe} of Emerton and Helm on Langlands correspondences \virg{in families}. Finally we shall mention the works of Gurevich and Hadani (see, for example, \cite{GuHa1} and \cite{GuHa2}) that generalize several constructions of \cite{Wei}, still remaining in the context of complex representations (i.e. with $R=\C$). 



\section{Notation and definitions}
Let $F$ be a locally compact non-archimedean field of characteristic different from $2$. We write $\OF$ for the ring of integers of $F$, we fix a uniformizer $\varpi$ of $\OF$, we denote $p$ the residue characteristic and $q$ the cardinality of the residue field of $F$. Let $R$ be an integral domain such that $p\in\Rx$. We assume that there exists a smooth non-trivial character $\chi:F\longrightarrow \Rx$, that is a group homomorphism from $F$ to $\Rx$ whose kernel is an open subgroup of $F$. These properties assure the existence of an integer $l=\min\{j\in \Z\,|\,\varpi^{j}\OF\subset \ker(\chi)\}$ called the \emph{conductor} of $\chi$.



\subsection*{Quadratic forms}

\noindent We denote by $G$ any finite dimensional vector space over $F$.

\vspace{0.2cm}
\noindent 
We recall that a \emph{quadratic form} on $G$ is a continuous map $f:G\rightarrow F$ such that $f(ux)=u^2f(x)$ for every $x\in G$ and $u\in F$ and $(x,y)\longmapsto f(x+y)-f(x)-f(y)$ is $F$-bilinear. A \emph{character of degree 2} of $G$ is a map $\varphi:G\rightarrow \Rx$ such that $(x,y)\longmapsto \varphi(x+y)\varphi(x)^{-1}\varphi(y)^{-1}$ is a bicharacter (i.e. a smooth character on each variable) of $G\times G$. 
We denote by $Q(G)$ the $F$-vector space of quadratic forms on $G$, by $X_2(G)$ the group of characters of degree 2 of $G$ endowed with the pointwise multiplication and by $X_1(G)$ the multiplicative group of smooth $R$-characters of $G$, that is a subgroup of $X_2(G)$.

\vspace{0.2cm}
\noindent 
We denote by $G^*=\Hom(G,F)$ the dual vector space of $G$. We write $[x,x^*]=x^*(x)\in F$ and $\la x,x^*\ra=\chi\left([x,x^*]\right)\in\Rx$ for every $x\in G$ and $x^*\in G^*$. We identify $(G^*)^*=G$ by means of $[x^*,x]=[x,x^*]$.
We have a group isomorphism
\begin{equation}\label{isomduali}
\begin{array}{ccl}
G^*&\longrightarrow &X_1(G)\\
x^*&\longmapsto & \la\, \cdot\,,x^*\ra.
\end{array}
\end{equation}
Indeed if $\la x,x^*\ra=1$ for every $x\in X$ then $[x,x^*]\in\ker(\chi)$ for every $x\in X$ and this implies that $x^*=0$ since $\ker(\chi)\neq F$. The surjectivity follows by Theorem II.3 of \cite{Wei2} and I.3.9 of \cite{Vig1}.

\begin{defin}\label{defB}
Let $\mathcal{B}$ be the bilinear map from $(G\times G^*)\times(G\times G^*)$ to $F$ defined by $\mathcal{B}\big((x_1,x_1^*),(x_2,x_2^*)\big)=[x_1,x_2^*]$
and let $\F=\chi\circ\mathcal{B}$.
\end{defin}

\noindent
For a $F$-linear map $\alpha:G\rightarrow H$ we denote by $\alpha^*:H^* \to G^*$ its \emph{transpose}. If $H=G^*$ and $\alpha=\alpha^*$ we say that $\alpha$ is \emph{symmetric}. We associate to every quadratic form $f$ on $G$ the symmetric homomorphism $\rho=\rho(f):G\rightarrow G^*$ defined by $\rho(x)(y)=f(x+y)-f(x)-f(y)$ for every $x,y\in G$. Since $\mathrm{char}(F)\neq 2$, the map $f\mapsto \rho(f)$ is an isomorphism from $Q(G)$ to the $F$-vector space of symmetric homomorphisms from $G$ to $G^*$ with inverse the map sending $\rho$ to the quadratic form $f(x)=[x,\frac{\rho(x)}{2}]$.
We say that $f\in Q(G)$ is \emph{non-degenerate} if $\rho(f)$ is an isomorphism and we denote by $Q^{nd}(G)$ the subgroup of $Q(G)$ of non-degenerate quadratic forms on $G$. We remark that the composition with the character $\chi$ gives an injective group homomorphism from $Q(G)$ to $X_2(G)$.


\subsection*{Integration theory}
\noindent
Let $dg$ be a Haar measure on $G$ with values in $R$ (see I.2 of \cite{Vig1}). 
We denote by $\S(G)$ the $R$-module of compactly supported locally constant functions on $G$ with values in $R$. We can write every $\Phi\in\S(G)$ as $\Phi=\sum_{h\in K_1/K_2}x_h \I_{h+K_2}$
where $K_1$ and $K_2$ are two 
compact open subgroups of $G$, $x_h\in R$, $\I_{h+K_2}$ is the characteristic function of $h+K_2$ and the sum is taken over the finite number of right cosets of $K_2$ in $K_1$.

\vspace{0,2cm}
\noindent
The \emph{Fourier transform} of $\Phi\in\S(G)$ is the function from $G^*$ to $R$ defined by
\begin{equation}\label{Fourier}
\Fou\Phi(g^*)=\int_G\Phi(g)\la g,g^*\ra dg
\end{equation}
for every $g^*\in G^*$. 

\vspace{0,2cm}
\noindent
For every compact open subgroup $K$ of $G$ let $K_*=\{g^*\in G^*\,|\,\la k,g^*\ra=1 \,\forall\, k\in K \}$ define a subgroup of $G^*$. Notice that the map $K\mapsto K_*$ is inclusion-reversing.\\ 
If $L$ is any $\OF$-lattice of $G$ and $l$ is the conductor of $\chi$, then $L_*=\{g^*\in G^*\,|\, g^*(L)\subset \varpi_F^l\OF \}$. Explicitly, if $L=\bigoplus_i\varpi_F^{a_i}\OF$ (with $a_i\in\Z$ for all $i$) with respect a fixed basis $(e_1,\dots,e_N)$ of $G$, then $L_*=\bigoplus_i\varpi_F^{l-a_i}\OF$ with respect to the dual basis of $(e_1,\dots,e_N)$ of $G^*$. 
These facts imply that $K_*$ is a compact open subgroup of $G^*$ for every compact open subgroup $K$ of $G$.

\vspace{0,2cm}
\noindent
Given a Haar measure $dg$ on $G$ such that $\mathrm{vol}(K',dg)=1$ we call \emph{dual measure of $dg$} the Haar measure $dg^*$ on $G^*$ such that $\mathrm{vol}(K'_*,dg^*)=1$.

\vspace{0,2cm}
\noindent
The \emph{inverse Fourier transform} of $\Psi\in\S(G^*)$ is the function from $G$ to $R$ defined by
\begin{equation}\label{Fourierinverse}
\Fou^{-1}\Psi(g)=\int_{G^*}\Psi(g^*)\la g,-g^*\ra dg^*
\end{equation}
for every $g\in G$. 

\vspace{0,2cm}
\noindent
For every $\Psi_1,\Psi_2\in\S(G^*)$, we denote by $\Psi_1*\Psi_2\in\S(G^*)$ the \emph{convolution product} defined by 
\[(\Psi_1*\Psi_2)(x^*)=\int_{G^*} \Psi_1(g^*)\Psi_2(x^*-g^*)dg^*\] for every $x^*\in G^*$.

\begin{prop}\label{teorFourier}
Formulas (\ref{Fourier}) and (\ref{Fourierinverse}) give an isomorphism of $R$-algebras from $\S(G)$, endowed with the pointwise product, to $\S(G^*)$, endowed with the convolution product.
\end{prop}

\begin{proof}
The $R$-linearity of $\Fou$ and $\Fou^{-1}$ is clear from their definitions. Let now $K$ be a compact open subgroup of $G$ and $h\in G$; we have that
$$\Fou \I_{h+K}(g^*)=\int_G\I_{K}(g-h)\la g,g^*\ra dg=\la h,g^*\ra\int_K\la g,g^*\ra dg.$$
Moreover we have $\int_K \la  g,g^*\ra dg=\la k,g^*\ra\int_K\la g,g^*\ra dg$ for every $k\in K$ and, since $R$ is an integral domain, we obtain that 
$\Fou \I_{h+K}(g^*)=\mathrm{vol}(K,dg)\la h,g^*\ra\I_{K_*}(g^*).$
Then $\Fou\Phi\in\S(G^*)$ for every $\Phi\in\S(G)$, since $\Fou$ is $R$-linear and $\Phi$ is a finite sum of the form $\sum_h x_h\I_{h+K_1}$ with $x_h\in R$ and $K_1$ a compact open subgroup of $G$.\\
Denoting $K_{**}=\{g\in G\,|\,\la g,g^*\ra \,\forall g^*\in K_*\}$ we have that 
\begin{align*}
\Fou^{-1}\Fou\I_{h+K}(g)&=\mathrm{vol}(K,dg)\int_{G^*}\la h,g^*\ra\I_{K_*}(g^*)\la g,-g^*\ra dg^*=\mathrm{vol}(K,dg)\int_{K_*}\la h-g,g^*\ra dg^*\\
&=\mathrm{vol}(K,dg)\mathrm{vol}(K_*,dg^*)\I_{h+K_{**}}.
\end{align*}
Moreover if $L=\bigoplus_i\varpi_F^{a_i}\OF$ is an $\OF$-lattice of $G$ as above then $L_{**}=\bigoplus_i\varpi_F^{l-(l-a_i)}\OF=L$. Let now $L$ be an $\OF$-lattice and $K$ be a compact open subgroup of $G$ such that $L\subset K$; we can write $\I_K=\sum_{h\in K/L}I_{h+L}$ and then we obtain
\begin{align*}
\Fou^{-1}\Fou\I_K&=\mathrm{vol}(K,dg)\mathrm{vol}(K_*,dg^*)\I_{K_{**}}\\
                 &=\mathrm{vol}(L,dg)\mathrm{vol}(L_*,dg^*)\sum_{h\in K/L}\I_{h+L_{**}}= \mathrm{vol}(L,dg)\mathrm{vol}(L_*,dg^*)\I_K.
\end{align*}
This implies that $K=K_{**}$ and $\mathrm{vol}(K,dg)\mathrm{vol}(K_*,dg^*)=1$ for every compact open subgroup $K$ of $G$. This proves that $\Fou$ is an isomorphism whose inverse is $\Fou^{-1}$.\\
Finally for every $\Psi_1,\Psi_2\in\S(G^*)$ we have 
\begin{align*}
\Fou^{-1}(\Psi_1*\Psi_2)(g)&=\int_{G^*}\int_{G^*}\Psi_1(g_1^*)\Psi_2(g_2^*-g_1^*)dg_1^*\la -g,g_2^*\ra dg_2^*
\\
&=\int_G\Psi_1(g_1^*)\int_G\Psi_2(g_3^*)\la -g,g_3^*+g_1^*\ra dg_3^*dg_1^*=\Fou^{-1}(\Psi_1)(g)\cdot\Fou^{-1}(\Psi_2)(g)
\end{align*}
where we have used the change of variables $g_2^*\longmapsto g_3^*=g_2^*-g_1^*$.
\end{proof}


\begin{defin}\label{module}
Let $G$ and $H$ be two finite dimensional $F$-vector spaces 
and let $dx$ and $dy$ be two Haar measures on $G$ and $H$. If $\nu:G\longrightarrow H$ is an isomorphism then the \emph{module} of $\nu$ is the constant $|\nu|=\frac{d(\nu x)}{dy}$, which means that we have
$$\int_{H}\Phi(y)dy=|\nu|\int_{G}\Phi(\nu(x))dx$$
where $\Phi\in\S(H)$. Notice that it is an integer power of $q$ in $R$.
\end{defin}

\noindent
If $dx^*$ and $dy^*$ are the dual measures on $G^*$ and $H^*$ of $dx$ and $dy$, then $|\nu|=|\nu^*|$ for every isomorphism $\nu:G\longrightarrow H$. Indeed if $K$ is a compact open subgroup of $G$ then 
$$\mathrm{vol}(K,dx)=|\nu|^{-1}\mathrm{vol}(\nu(K),dy)=|\nu|^{-1}\mathrm{vol}(\nu(K)_{*},dy^*)^{-1}=|\nu|^{-1}|\nu^*|\mathrm{vol}(\nu^{*}(\nu(K))_{*},dx^*)^{-1}$$
and $\nu^{*}(\nu(K))_{*}=\{g^*\in G^*\,|\,\la\nu(k),\nu^{*-1}(g^*)\ra=1 \; \forall \; k\in K\}=K_*$. Then $|\nu|=|\nu^*|$.\\
Moreover if $G=H$ and $dx=dy$ we have that $|\nu|$ is independent of the choice of the Haar measure $dx$ on $G$.


\subsection*{The symplectic group}
\noindent From now on, let $X$ be a finite dimensional $F$-vector space and let $W$ be the $F$-vector space $X\times X^*$. We denote by $\Sp$ the group of symplectic automorphisms of $W$, said to be the \emph{symplectic group} of $W$, that is the group of automorphisms of $W$ such that 
\begin{equation}\label{condsimpl}
\mathcal{B}\big(\sigma(w_1),\sigma(w_2)\big)-\mathcal{B}\big(\sigma(w_2),\sigma(w_1)\big)=\mathcal{B}(w_1,w_2)-\mathcal{B}(w_2,w_1),
\end{equation}
or equivalently, by (\ref{isomduali}), such that
$\mathcal{F}\big(\sigma(w_1),\sigma(w_2)\big)\mathcal{F}\big(\sigma(w_2),\sigma(w_1)\big)^{-1}=\mathcal{F}(w_1,w_2)\mathcal{F}(w_2,w_1)^{-1}.$

\vspace{0,2cm}
\begin{prop}\label{simplFlineare}
Every group automorphism $\sigma:W\longrightarrow W$ which satisfies (\ref{condsimpl}) is $F$-linear.
\end{prop}

\begin{proof}
Applying the change of variables $w_1\mapsto uw_1$ with $u\in F$ in the equality (\ref{condsimpl}), we obtain
$\mathcal{B}\big(\sigma(uw_1),\sigma(w_2)\big)-\mathcal{B}\big(\sigma(w_2),\sigma(uw_1)\big)=u\big(\mathcal{B}(w_1,w_2)-\mathcal{B}(w_2,w_1)\big)$
and then using (\ref{condsimpl}) again we obtain
$\mathcal{B}\big(\sigma(uw_1)-u\sigma(w_1),\sigma(w_2)\big)=\mathcal{B}\big(\sigma(w_2),\sigma(uw_1)-u\sigma(w_1)\big)$ for every $w_1,w_2\in W$.
This implies that $\mathcal{B}\big(\sigma(uw_1)-u\sigma(w_1),\sigma(w_2)\big)=0$ for every $w_2\in\sigma^{-1}(0\times X^*)$ and $\mathcal{B}\big(\sigma(w_2),\sigma(uw_1)-u\sigma(w_1)\big)=0$ for every $w_2\in\sigma^{-1}(X\times 0)$. Then $\sigma(uw_1)=u\sigma(w_1)$ for every $w_1\in W$. 
\end{proof}

\noindent
We can write every $\sigma\in \Sp$ as a matrix of the form $\begin{pmatrix}\alpha&\beta\\ \gamma&\delta\end{pmatrix}$ where $\alpha:X\rightarrow X$, $\gamma:X\rightarrow X^*$, $\beta:X^*\rightarrow X$ and $\delta:X^*\rightarrow X^*$ are $F$-linear. 
The transpose of $\sigma$ is $\sigma^*=\begin{pmatrix}\alpha^*&\gamma^*\\ \beta^*&\delta^*\end{pmatrix}$ which is an automorphism of $W^*=X^*\times X$ such that $|\sigma^*|=|\sigma|$. Furthermore if $\xi:X\times X^*\longrightarrow X^*\times X$ is the isomorphism defined by $(x,x^*)\longmapsto (-x^*,x)$ and $\sigma^I=\xi^{-1}\sigma^*\xi=\begin{pmatrix}\delta^*&-\beta^*\\ -\gamma^*&\alpha^*\end{pmatrix}$, then we have $|\sigma|=|\sigma^I|$.
With these definitions, an element $\sigma\in\Aut(W)$ is symplectic if and only if $\sigma^I\sigma=1$ and then the module of every symplectic automorphism is equal to $1$.\\
Moreover we can remark that if $\sigma\in\Sp$ then $\alpha^*\gamma=\gamma^*\alpha:X\longrightarrow X^*$ and $\beta^*\delta=\delta^*\beta:X^*\longrightarrow X$ are symmetric homomorphisms and $\alpha^*\delta-\gamma^*\beta=1$ and $\delta^*\alpha-\beta^*\gamma=1$.

\vspace{0.25cm}
\noindent
We associate to every $\sigma\in \Sp$ the quadratic form defined by 
$$f_{\sigma}(w)=\frac{1}{2}\big(\mathcal{B}(\sigma(w),\sigma(w))-\mathcal{B}(w,w)\big).$$ It is easy to check that $f_{\sigma_1\circ\sigma_2}=f_{\sigma_1}\circ\sigma_2+f_{\sigma_2}$ for every $\sigma_1,\sigma_2\in\Sp$ and that 
\begin{equation}\label{pseudo}
f_{\sigma}(w_1+w_2)-f_{\sigma}(w_1)-f_{\sigma}(w_2)=\mathcal{B}(\sigma(w_1),\sigma(w_2))-\mathcal{B}(w_1,w_2)
\end{equation}
for every $\sigma\in\Sp$ and $w_1,w_2\in W$.


\subsubsection*{Symplectic realizations of forms}
\noindent
We introduce some applications, similar to those in 33 of \cite{Wei}, with values in $\Sp$ and we give some relations between them. When comparing our calculations with those of sections 6 and 7 of \cite{Wei} it shall be remarked that we change most of the definitions because we consider matrices acting on the left rather than on the right, to uniform notation to the contemporary standard. This affects also the formulas that explicit the relations between these applications.

\begin{defin}\label{applSp}
We define the following maps.
\begin{itemize}
		\item An injective group homomorphism from $\Aut(X)$ to $\Sp$:
\begin{equation*}
\begin{array}{rcl}
d:\Aut(X)&\longrightarrow &\Sp\\
\alpha&\longmapsto &\begin{pmatrix}\alpha&0\\0&\alpha^{*-1}\end{pmatrix}.
\end{array}
\end{equation*}
		\item An injective map from $\Isom(X^*,X)$ to $\Sp$ where $\Isom(X^*,X)$ is the set of isomorphisms from $X^*$ to $X$:
\begin{equation*}
\begin{array}{rcl}
d':\Isom(X^*,X)&\longrightarrow &\Sp\\
\beta&\longmapsto &\begin{pmatrix}0&\beta\\-\beta^{*-1}&0\end{pmatrix}.
\end{array}
\end{equation*}
We remark that $d'(\beta)^{-1}=d'(-\beta^*)$ for every $\beta\in\Isom(X^*,X)$.
		\item An injective group homomorphism from $Q(X)$ to $\Sp$:
\begin{equation*}
\begin{array}{rcl}
t:Q(X)&\longrightarrow &\Sp\\
f&\longmapsto &\begin{pmatrix}1&0\\-\rho&1\end{pmatrix}
\end{array}
\end{equation*}
where $\rho=\rho(f)$ is the symmetric homomorphism associated to $f$.
		\item An injective group homomorphism from $Q(X^*)$ to $\Sp$:
\begin{equation*}
\begin{array}{rcl}
t':Q(X^*)&\longrightarrow &\Sp\\
f'&\longmapsto &\begin{pmatrix}1&-\rho'\\0&1\end{pmatrix}
\end{array}
\end{equation*}
where $\rho'=\rho(f')$ is the symmetric homomorphism associated to $f'$.
\end{itemize}
\end{defin}

\bigskip
\noindent Let $G$ be either $X$ or $X^*$. If $f\in Q(G)$ and $\alpha\in \Aut(G)$ we write $f^{\alpha}$ for $f\circ\alpha$.

\begin{prop}\label{relazSp}\mbox{}
\begin{enumerate}
	\item[(i)] Let $f\in Q(X)$, $f'\in Q(X^*)$ and $\alpha\in \Aut(X)$. Then 
	$d(\alpha)^{-1}t(f)d(\alpha)=t(f^{\alpha})$ and $d(\alpha)t'(f')d(\alpha)^{-1}=t'(f'^{\alpha^*}).$
	\item[(ii)] Let $\alpha\in \Aut(X)$, $\beta\in \Isom(X^*,X)$. Then 
	$d'(\alpha\beta)=d(\alpha)d'(\beta)$ and $d'(\beta\alpha^{*-1})=d'(\beta)d(\alpha).$
\end{enumerate}
\end{prop}


\begin{proof}\mbox{}
\begin{enumerate}
	\item[(i)] We have
$d(\alpha)^{-1}t(f)d(\alpha)=\begin{pmatrix}\alpha^{-1}&0\\0&\alpha^{*}\end{pmatrix}\begin{pmatrix}1&0\\-\rho&1\end{pmatrix}\begin{pmatrix}\alpha&0\\0&\alpha^{*-1}\end{pmatrix}=\begin{pmatrix}1&0\\ -\alpha^{*}\rho\alpha&1\end{pmatrix}.$
It is easy to check that the symmetric homomorphism associated to $f^{\alpha}$ is $-\alpha^{*}\rho\alpha$. With similar explicit calculations the second equality can be proven as well.
	\item[(ii)] We have 
	$d(\alpha)d'(\beta)=\begin{pmatrix}\alpha&0\\0&\alpha^{*-1}\end{pmatrix}\begin{pmatrix}0&\beta\\ -\beta^{*-1}&0\end{pmatrix}= \begin{pmatrix}0&\alpha\beta\\-\alpha^{*-1}\beta^{*-1}&0\end{pmatrix}=d'(\alpha\beta),
$
and
$d'(\beta)d(\alpha)=\begin{pmatrix}0&\beta\\ -\beta^{*-1}&0\end{pmatrix}\begin{pmatrix}\alpha&0\\0&\alpha^{*-1}\end{pmatrix}
=\begin{pmatrix}0&\beta\alpha^{*-1}\\-(\beta\alpha^{*-1})^{*-1}&0\end{pmatrix}=d'(\beta\alpha^{*-1}).$\qedhere
\end{enumerate}
\end{proof}

\noindent
We have $d(\alpha)d'(\beta)d(\alpha)^{-1}=d'(\alpha\circ\beta\circ\alpha^{*})$ so that the group $d(\Aut(X))$ acts on the set $d'(\Isom(X^*, X))$ by conjugacy in $\Sp$.


\subsubsection*{A set of generators for the symplectic group}
Let us provide a description of $\Sp$ by generators and relations. We denote by $\Omega(W)$ the subset of $\Sp$ of elements $\sigma=\begin{pmatrix}\alpha&\beta\\ \gamma&\delta\end{pmatrix}$ such that $\beta$ is an isomorphism. The set $\Omega(W)$ is a set of generators for $\Sp$ (cf. 42 of \cite{Wei}). The precise statement is as follows.

\begin{prop}\label{lemma6}
The group $\Sp$ is generated by the elements of $\Omega(W)$ with relations $\sigma\sigma'=\sigma''$ for every $\sigma,\sigma',\sigma''\in\Omega(W)$ such that the equality $\sigma\sigma'=\sigma''$ holds in $\Sp$. 
\end{prop}
\noindent
Weil states also the following fact about the set $\Omega(W)$ (cf. formula (33) of \cite{Wei}). 

\begin{prop}\label{omegaSp}
Every element $\sigma\in\Omega(W)$ can be written as $\sigma=t(f_1)d'(\beta')t(f_2)$
for unique $f_1,f_2\in Q(X)$ and $\beta'\in\Isom(X^*,X)$.
\end{prop}


\begin{rmk}\label{dimOmegaSp}
\noindent
Let $\begin{pmatrix}\alpha&\beta\\ \gamma&\delta\end{pmatrix}\in\Omega(W)$. Then $\sigma=t(f_1)d'(\beta)t(f_2)$ where $f_1$ and $f_2$ are the quadratic forms associated to the symmetric homomorphisms $-\delta\beta^{-1}$ and $-\beta^{-1}\alpha$. In particular we have the formula
$$\begin{pmatrix}\alpha&\beta\\ \gamma&\delta\end{pmatrix}=\begin{pmatrix}1&0\\ \delta\beta^{-1}&1\end{pmatrix}\begin{pmatrix}0&\beta\\ -\beta^{*-1}&0\end{pmatrix}\begin{pmatrix}1&0\\ \beta^{-1}\alpha&1\end{pmatrix}.$$
\end{rmk}


\section{The metaplectic group}

Following Weil's strategy we define the metaplectic group, attached to $R$ and $\chi$, as a central extension of the symplectic group by $\Rx$. To do so, we shall construct the groups $\B$ and $\BB$. In particular, in Theorem \ref{ses1} we characterize $\BB$ as central extension of  $\B$ by $R^\times$. This characterization permits to define the metaplectic group as fiber product over $\B$ of the symplectic group and $\BB$ and to show that the metaplectic group is a central extension of the symplectic group by $\Rx$.\\ The main issue related to this group, rather than its formal definition, is to study the maps $\mu: \Sp\to \B$ and $\pi_0: \BB \to \B$, that depend both on $R$.


\subsection{The group $B_0(W)$}\label{defB0}
\noindent
Let $A(W)$ be the group whose underlying set is $W\times \Rx$ with the multiplication law
$$(w_1,t_1)(w_2,t_2)=(w_1+w_2,t_1t_2\F(w_1,w_2))$$
where $\F$ is as in Definition \ref{defB}. Its center is $Z=Z(A(W))=\{(0,t), t\in \Rx\}\cong\Rx$. 

\vspace{0.25cm}
\noindent We denote by $B_0(W)$ the subgroup of $\Aut(A(W))$ of group automorphisms of $A(W)$ acting trivially on $Z$, i.e. $B_0(W)=\{s\in\Aut(A(W))\;|\; s_{|Z}=\mathrm{id}_Z\}.$

\begin{prop}\label{propB0}
Let $s\in B_0(W)$. Then there exists a unique pair $(\sigma,\varphi)\in\Sp\times X_2(W)$ satisfying the property
\begin{equation}\label{condB0}
\varphi(w_1+w_2)\varphi(w_1)^{-1}\varphi(w_2)^{-1}=\F\big(\sigma(w_1),\sigma(w_2)\big)\F(w_1,w_2)^{-1}
\end{equation}
such that $s(w,t)=(\sigma(w),\varphi(w)t)$ for every $w\in W$ and $t\in\Rx$. 
Conversely if the pair $(\sigma,\varphi)\in\Sp\times X_2(W)$ satisfies (\ref{condB0}), then $(w,t) \mapsto (\sigma(w),\varphi(w)t)$ defines an element of $B_0(W)$.
\end{prop}

\begin{proof} 
Let $\eta:A(W)\longrightarrow W$ and $\theta:A(W)\longrightarrow \Rx$ such that $s(w,t)=(\eta(w,t),\theta(w,t))$. For every $w_1,w_2\in W$ and $t_1,t_2\in\Rx$ we have
\begin{align*}
s((w_1,t_1)(w_2,t_2))&=\big(\eta(w_1+w_2,t_1t_2\F(w_1,w_2)),\theta(w_1+w_2,t_1t_2\F(w_1,w_2))\big)\\
s(w_1,t_1)s(w_2,t_2) &=\big(\eta(w_1,t_1)+\eta(w_2,t_2),\theta(w_1,t_1)\theta(w_2,t_2)\F(\eta(w_1,t_1),\eta(w_2,t_2))\big).
\end{align*}
Since $s$ is a homomorphism then $\eta$ is so and since $s_{|Z}=\mathrm{id}_Z$ then $\eta(0,t)=0$ for every $t\in \Rx$. These two facts imply that $\eta(w,t)=\eta(w,1)$ for every $t\in \Rx$ so that $\sigma$, defined by $\sigma(w)=\eta(w,1)$, is a group endomorphism of $W$. We have also
\begin{equation}\label{eq1}
\theta(w_1+w_2,t_1t_2\F(w_1,w_2))=\theta(w_1,t_1)\theta(w_2,t_2)\F(\sigma(w_1),\sigma(w_2)).
\end{equation}
Setting $w_2=0$ and $t_1=1$ and using the fact that $\theta(0,t)=t$ for every $t\in \Rx$ (since $s_{|Z}=\mathrm{id}_Z$) we obtain that $\theta(w_1,t_2)=\theta(w_1,1)t_2$ for every $w_1\in W$ and $t_2\in \Rx$. So, if we set $\varphi(w)=\theta(w,1)$, we obtain that $s(w,t)=(\sigma(w),\varphi(w)t)$ and (\ref{eq1}) becomes
\begin{equation*}
\varphi(w_1+w_2)t_1t_2\F(w_1,w_2)=\varphi(w_1)t_1\varphi(w_2)t_2\F(\sigma(w_1),\sigma(w_2))
\end{equation*}
that is exactly the condition (\ref{condB0}). Furthermore, if we take $\sigma'\in\End(W)$ and $\varphi':W\longrightarrow \Rx$ such that $s^{-1}(w,t)=(\sigma'(w),\varphi'(w)t)$, then $(w,t)=s(s^{-1}(w,t))=(\sigma(\sigma'(w)),\varphi(\sigma'(w))\varphi'(w)t)$ that implies that $\sigma$ is a group automorphism of $W$ with $\sigma^{-1}=\sigma'$.
Now, the left-hand side of (\ref{condB0}) is symmetric on $w_1$ and $w_2$, so $\sigma$ verify the symplectic property and by Proposition \ref{simplFlineare}, $\sigma\in\Sp$. Furthermore the right-hand side of (\ref{condB0}) is a bicharacter and so $\varphi$ is a character of degree 2 of $W$.\\
For the vice-versa, it is easy to check that $(w,t) \mapsto (\sigma(w),\varphi(w)t)$ is an endomorphism of $A(W)$ thanks to the property (\ref{condB0}), and that it is invertible with inverse $(w,t) \mapsto (\sigma^{-1}(w),(\varphi(\sigma^{-1}w))^{-1}t)$. Notice that it acts trivially on $Z$, so it is an element of $\B$.
\end{proof}

\noindent From now on, we identify an element $s\in\B$ with the corresponding pair $(\sigma,\varphi)$ such that $s(w,t)=(\sigma(w),\varphi(w)t)$.
If $s_1,s_2\in B_0(W)$ and $(\sigma_1,\varphi_1)$ and $(\sigma_2,\varphi_2)$ are their corresponding pairs, then the composition law of $\B$ becomes
$s_1\circ s_2=(\sigma_1,\varphi_1)(\sigma_2,\varphi_2)=(\sigma_1\circ \sigma_2,\varphi)$ where $\varphi$ is defined by $\varphi(w)=\varphi_2(w)\varphi_1(\sigma_2(w))$. We observe that the identity element is $(\mathrm{id},1)$ and the inverse of $(\sigma,\varphi)$ is $(\sigma^{-1},(\varphi\circ\sigma^{-1})^{-1})$.

\vspace{0,2cm}
\noindent
The projection $\pi':\B\longrightarrow \Sp$ defined by $\pi'(\sigma,\varphi)=\sigma$ is a group homomorphism whose kernel is $\{(\mathrm{id},\tau), \tau\in X_1(W)\}$. Furthermore, by (\ref{pseudo}) and (\ref{condB0}), we have an injective group homomorphism
\begin{equation}\label{defmu}
\begin{array}{rcl}
\mu:\Sp&\longrightarrow & \B\\
\sigma&\longmapsto & (\sigma,\chi\circ f_{\sigma})
\end{array}
\end{equation}
such that $\pi'\circ\mu$ is the identity of $\Sp$. This means that $\B$ is the semidirect product of $\{(\mathrm{id},\tau), \tau\in X_1(W)\}$ and $\mu(\Sp)$ and in particular, by Propositions \ref{lemma6} and \ref{omegaSp}, it is generated by $\mu(t(Q(X)))$, $\mu(d'(\Isom(X^*,X)))$ and $\{(\mathrm{id},\tau), \tau\in X_1(W)\}$.

\bigskip

\noindent
Let us define some applications with values in $\B$, similar to those in 6 of \cite{Wei}, composing those with values in $\Sp$ with $\mu$. We call them $d_0=\mu\circ d$, $d'_0=\mu\circ d'$, $t_0=\mu\circ t$ and $t'_0=\mu\circ t'$.


\subsection{The group $\BB$}\label{defBB}

We define $\A$ as the image of a faithful infinite dimensional representation of $A(W)$ over $R$ and $\BB$ as its normalizer in $\Aut(\S(X))$. Then we show that in fact $\BB$ is a central extension of $\B$ by $\Rx$.

\subsubsection{$\A$ and $\BB$}
\noindent
For every $w=(v,v^*)\in X\times X^*= W$ and every $t\in \Rx$, we denote by $U(w,t)$ the $R$-linear operator on $\S(X)$ defined by $$U(w,t)\Phi:x\mapsto t\Phi(x+v)\la x,v^*\ra$$ for every function $\Phi\in\S(X)$. It can be directly verified that $U(w,t)$ lies in $\Aut(\S(X))$ for every $w\in W$ and $t\in\Rx$. With a slight abuse of notation we write $U(w)=U(w,1)$ for every $w\in W$.

\smallskip

\noindent
Let $\A=\left\{U(w,t)\in\Aut(\S(X))\;|\;t\in\Rx, w\in W\right\}$. It is not hard to see that it is a subgroup of $\Aut(\S(X))$ and that its multiplication law is given by 
\begin{equation}\label{operazioneA}
U(w_1,t_1)U(w_2,t_2)=U(w_1+w_2,t_1t_2\F(w_1,w_2)).
\end{equation}

\begin{lemma}\label{isomA} The map
\begin{equation*}
\begin{array}{rcl}
U:A(W)&\longrightarrow &\A\\
(w,t)&\longmapsto &U(w,t).
\end{array}
\end{equation*}
is a group isomorphism. 
\end{lemma}

\begin{proof}
By (\ref{operazioneA}) the map $U$ preserves operations and it is cleraly surjective. For injectivity we have to prove that if $t\Phi(x+v)\la x,v^*\ra=\Phi(x)$ for every $\Phi\in\S(X)$ and every $x\in X$ then $t=1$ and $(v,v^*)=(0,0)$. If we take $x=0$ and $\Phi$ the characteristic function $\I_K$ of any compact open subgroup $K$ of $X$, we obtain that $t\I_K(v)=1$ for every $K$ and so $t=1$ and $v=0$. Therefore we have that $\la x,v^*\ra=1$ for every $x\in X$ and so $v^*=0$ by (\ref{isomduali}).
\end{proof}

\begin{rmk}
The homomorphism $U$ is a representation of $A(W)$ on the $R$-module $\S(X)$.
\end{rmk}

\noindent
The group $\B$ acts on $A(W)$ and so on $\A$ via the isomorphism in Lemma \ref{isomA}. This action is given by
\begin{equation*}
\begin{array}{rcl}
\B\times\A&\longrightarrow &\A\\
((\sigma,\varphi),U(w,t))&\longmapsto &U(\sigma(w),t\varphi(w)).
\end{array}
\end{equation*}
Moreover, we can identify $\B$ with the group of automorphisms of $\A$ acting trivially on the center $Z(\A)=\{t\cdot\mathrm{id}_{\S(X)}\in\Aut(\S(X))\;|\;t\in\Rx\}\cong \Rx$.

\bigskip

\noindent

We denote by $\BB$ the normalizer of $\A$ in $\Aut(\S(X))$, that is $$\BB=\left\{\mathbf{s}\in\Aut(\S(X))\;|\;\mathbf{s}\A\mathbf{s}^{-1}=\A\right\}.$$
So, if $\mathbf{s}$ is an element of $\BB$, conjugation by $\mathbf{s}$, denoted by $\mathrm{conj}(\mathbf{s})$, is an automorphism of $\A$.

\begin{lemma}\label{defpi}
The map
\begin{equation*}
\begin{array}{rcl}
\pi_0:\BB&\longrightarrow &\B\\
\mathbf{s}&\longmapsto &\mathrm{conj}(\mathbf{s})
\end{array}
\end{equation*}
is a group homomorphism
\end{lemma}

\begin{proof}
Clearly $\mathrm{conj}(\mathbf{s})$ is trivial on $Z(\A)=\{t\cdot\mathrm{id}_{\S(X)}\in\Aut(\S(X))\;|\;t\in\Rx\}$ and so it lies in $\B$. Moreover $\mathrm{conj}(\mathbf{s}_1\mathbf{s}_2)=\mathrm{conj}(\mathbf{s}_1)\mathrm{conj}(\mathbf{s}_2)$ so that $\pi_0$ preserves the group operation.
\end{proof}

\begin{teor}\label{ses1}
The following sequence is exact: $$1\longrightarrow \Rx \longrightarrow \BB \stackrel{\pi_0}{\longrightarrow}B_0(W) \longrightarrow 1$$ where $\Rx$ injects in $\B$ by $t\mapsto t\cdot\mathrm{id}_{\S(X)}$.
\end{teor}

\noindent
We prove this theorem in paragraph \ref{provases1}. Before that, we need to construct, as proposed in 13 of \cite{Wei}, some \virg{liftings} to $\BB$ of the applications $d_0$, $d_0'$ and $t_0$.


\subsubsection{Realization of forms on $\BB$}
We fix a Haar measure $dx$ on the finite dimensional $F$-vector space 
$X$ with values in $R$. We denote by $dx^*$ the dual measure of $dx$ on $X^*$ and $dw=dxdx^*$ the product Haar measure on $W$.

\vspace{0.2cm}
\noindent
From now on, we suppose that there exists a fixed square root $q^{\frac{1}{2}}$ of $q$ in $R$. If $\nu$ is an isomorphism of $F$-vector spaces and $|\nu|=q^a$ is its module, we denote $|\nu|^{\frac{1}{2}}=(q^{\frac{1}{2}})^a\in R$. 

\begin{defin}\label{applBB}
We define the following maps.
\begin{itemize}
		\item A group homomorphism $\d_0:\Aut(X)\longrightarrow \Aut(\S(X))$ defined by
$\d_0(\alpha)\Phi= |\alpha|^{-\frac{1}{2}}(\Phi\circ\alpha^{-1})$
for every $\alpha\in\Aut(X)$ and every $\Phi\in\S(X)$.
		\item A map $\d_0':\Isom(X^*,X)\longrightarrow \Aut(\S(X))$ defined by
$\d'_0(\beta)\Phi=|\beta|^{-\frac{1}{2}}(\Fou\Phi\circ\beta^{-1})$
for every $\beta\in\Isom(X^*,X)$ and every $\Phi\in\S(X)$, where $\Fou\Phi$ is the Fourier transform of $\Phi$ as in (\ref{Fourier}). We remark that $\d'_0(\beta)^{-1}=\d'_0(-\beta^*)=|\beta|^{\frac{1}{2}}\Fou^{-1}(\Phi\circ\beta)$.
	\item A group homomorphism $\t_0:Q(X)\longrightarrow \Aut(\S(X))$ defined by
$\t_0(f)\Phi= (\chi\circ f)\cdot\Phi$
for every $f\in Q(X)$ and every $\Phi\in\S(X)$.
\end{itemize}
\end{defin}

\noindent
We shall now to prove that they are actually onto $\BB$ and that they lift in $\BB$ the applications $d_0$, $d_0'$ and $t_0$.
\begin{prop}\label{rialz}
The images of $\d_0$, $\d_0'$ and $\t_0$ are in $\BB$ and they satisfy
$$\pi_0\circ\d_0=d_0 \qquad \pi_0\circ\d'_0=d'_0 \qquad \mbox{and} \qquad \pi_0\circ\t_0=t_0.$$
\end{prop}

\begin{proof}
For every $\alpha\in \Aut(X)$, $\Phi\in \S(X)$, $w=(v,v^*)\in W$ and $x\in X$ we have
\begin{align*}
\d_0(\alpha)U(w)\d_0(\alpha)^{-1}\Phi(x)&=\d_0(\alpha)U(w)|\alpha|^{\frac{1}{2}}(\Phi\circ \alpha)(x)=\Phi(\alpha(\alpha^{-1}(x)+u))\la \alpha^{-1}(x),v^*\ra\\
																				&=\Phi(x+\alpha(u))\la x,\alpha^{*-1}(v^*)\ra= d_0(\alpha)U(w)\Phi(x).
\end{align*}
For every $\beta\in \Isom(X^*,X)$, $\Phi\in \S(X)$, $w=(v,v^*)\in W$ and $x\in X$ we have
\begin{align*}
\d'_0(\beta)U(w)\d'_0(\beta)^{-1}\Phi(x)&=\d'_0(\beta)U(w)|\beta|^{\frac{1}{2}}\Fou^{-1}(\Phi\circ\beta)(x)\\
																				&=\int_X\left(\int_{X^*}\Phi(\beta(x^*))\la x_1+v,-x^*\ra dx^*\right) \la x_1,v^*\ra \la x_1,\beta^{-1}(x)\ra dx_1\\
																				&=\int_X\left(\int_{X^*}\Phi(\beta(x^*))\la -v,x^*\ra \la x_1,-x^*\ra dx^* \right)\la x_1,v^*+\beta^{-1}(x)\ra dx_1\\
																				&=\Phi(\beta(v^*+\beta^{-1}(x)))\la -v,v^*+\beta^{-1}(x)\ra\\
																				&=\Phi(x+\beta(v^*))\la x,-\beta^{*-1}(v)\ra\la v,-v^*\ra=d'_0(\beta)U(w)\Phi(x).
\end{align*}
For every $f\in Q(X)$, $\Phi\in \S(X)$, $w=(v,v^*)\in W$ and $x\in X$ we have
\begin{align*}
\t_0(f)U(w)\t_0(f)^{-1}\Phi(x)&=\chi(f(x))\chi(f(x+v))^{-1}\Phi(x+v)\la x,v^*\ra\\
																					&=\chi(f(v))^{-1}\la x,\rho(v)\ra^{-1} \Phi(x+v)\la x,v^*\ra
																					=t_0(f)U(w)\Phi(x).
\end{align*}
These equalities prove at the same time that the images of $\d_0$, $\d_0'$ and $\t_0$ are in $\BB$ and that they lift in $\BB$ respectively the applications $d_0$, $d_0'$ and $t_0$.
\end{proof}

\noindent
Proposition \ref{rialz} and the injectivity of $d_0$ and $t_0$ entail injectivity for $\d_0$ and $\t_0$.
Moreover Propositions \ref{relazSp} and \ref{rialz} say that for every $f\in Q(X)$, $\alpha\in \Aut(X)$ and $\beta\in \Isom(X^*,X)$, the three elements $\d_0(\alpha)^{-1}\t_0(f)\d_0(\alpha)$, $\d'_0(\alpha\circ\beta)$ and $\d'_0(\beta\circ\alpha^{*-1})$ of $\BB$ differ, respectively from $\t_0(f^{\alpha})$, $\d_0(\alpha)\d'_0(\beta)$ and $\d'_0(\beta)\d_0(\alpha)$ just by elements of $\Rx$.  A direct calculation gives
\begin{equation}\label{relazBB}
\d_0(\alpha)^{-1}\t_0(f)\d_0(\alpha)=\t_0(f^{\alpha}) \quad\; \d'_0(\alpha\circ\beta)=\d_0(\alpha)\d'_0(\beta) \quad\; \d'_0(\beta\circ\alpha^{*-1})=\d'_0(\beta)\d_0(\alpha)
\end{equation}
so that in fact these elements are the identity.


\subsubsection{Proof of Theorem \ref{ses1}}\label{provases1}
\noindent
In this paragraph we give a proof of Theorem \ref{ses1} that is fundamental for the definition of the metaplectic group.

\vspace{0,2cm}
\noindent
Firstly we prove that $\pi_0$ is surjective: we know that $\B$ is generated by $\mu(t(Q(X)))$, $\mu(d'(\Isom(X^*,X)))$ and $\{(\mathrm{id},\tau), \tau\in X_1(W)\}$ so that it is sufficient to prove that every element in these sets is in the image of $\pi_0$. By Proposition \ref{rialz}, this is proved for the sets $\mu(t(Q(X)))$ and  $\mu(d'(\Isom(X^*,X)))$. Moreover by (\ref{isomduali}) we have that every character $\tau$ of $W$ is of the form $\tau(v,v^*)=\la a,v^*\ra \la v,a^*\ra$ for suitable $a\in X$ and $a^*\in X^*$. For every $w=(v,v^*)\in W$ and $t\in \Rx$ we have
$(1,\tau)U(w,t)=U(w,t\cdot\tau(w))=U(w,t\la a,v^*\ra \la v,a^*\ra)=U(a,-a^*)U(w,t)U(-a,a^*,\la a,-a^*\ra)$
and so $(\mathrm{id},\tau)=\pi_0(U(a,-a^*))$.

\vspace{0,2cm}
\noindent
Let us now calculate the kernel of $\pi_0$. For $\phi\in\S(X\times X^*)$ we denote by $\mathcal U(\phi)$ the operator on $\S(X)$ defined by
$$ \mathcal U(\phi)=\int_W U(w,\phi(w))dw=\int_W \phi(w)U(w)dw.$$
This means that for every $\Phi\in\S(X)$ and every $x\in X$ we have
$$ \mathcal U(\phi)\Phi(x)=\int_W \phi(w)(U(w)\Phi)(x)dw=\int_W \phi(v,v^*)\Phi(x+v)\la x,v^*\ra  dvdv^*$$      
where $w=(v,v^*)$. Given $P,Q\in\S(X)$ we denote by $\phi_{P,Q}\in\S(X\times X^*)$ the function defined by
$$\phi_{P,Q}(v,v^*)=\int_XP(v')Q(v'+v)\la -v',v^*\ra dv'$$
for every $v\in X,v^*\in X^*$. With this definition we obtain
$$\mathcal U(\phi_{P,Q})\Phi(x)=\int_X\Phi(x+v)\int_{X^*}\int_X  P(v')Q(v'+v)\la x-v',v^*\ra dv'dv^*dv$$
and using Proposition \ref{teorFourier} we have
$$\mathcal U(\phi_{P,Q})\Phi(x)= \int_X\Phi(x+v) P(x)Q(x+v)dv=\int_X\Phi(v)Q(v)dv P(x).$$

\noindent If we denoted by $[P,Q]=\int_X P(x)Q(x)dx$ for every $P,Q\in\S(X)$ we have $\mathcal U(\phi_{P,Q})\Phi=[\Phi,Q]P$.

\noindent
Now, $\s$ is in the kernel of $\pi_0$ if and only if it lies in the centralizer of $\A$ in $\Aut(\S(X))$. If this is the case, then $\s$ commutes with $ \mathcal U(\phi)$ in $\End(\S(X))$ for every $\phi\in\S(X\times X^*)$, i.e. $\s (U(\phi)\Phi)= U(\phi)(\s(\Phi))$. In particular $\s$ commutes with operators of the form $ \mathcal U(\phi_{P,Q})$ for every $P,Q\in\S(X)$, that is $[\s\Phi,Q]P=[\Phi,Q]\s P$ for every $\Phi,P,Q\in\S(X)$. If we choose $\Phi= Q=\I_K$ where $K$ is a compact open subgroup of $X$ with $\mathrm{vol}(K,dx)\in\Rx$, we can write $$\s P= \frac{[\s\Phi,Q]}{[\Phi,Q]}P.$$ In other words $\s$ is of the form $\Phi\mapsto t\Phi$ for a suitable $t\in R$ and $t$ has to be invertible since $\s$ is an automorphism. Hence $\ker(\pi_0)\subseteq\{t\cdot\mathrm{id}_{\S(X)}\in\Aut(\S(X))\,|\,t\in\Rx\}$. The converse is true because the center of a group is always contained in its centralizer.

\begin{rmk}
In proving Theorem \ref{ses1} the techniques used in \cite{Wei} could be adapted to show that $\ker(\pi_0)\cong \Rx$, but not to prove surjectivity of $\pi_0$.
\end{rmk}


\subsection{The metaplectic group}
\noindent We have just defined in (\ref{defmu}) and Lemma \ref{defpi} the group homomorphisms
\begin{equation*}
\begin{array}{rclcrcl}
\mu:\Sp&\longrightarrow & \B                 &\quad\text{and}\quad & \pi_0:\BB&\longrightarrow & \B \\
\sigma&\longmapsto & (\sigma,\chi\circ f_{\sigma})  &\qquad &\mathbf{s}&\longmapsto & conj(\mathbf{s}).
\end{array}
\end{equation*}
The first one is injective, while the second one is surjective with kernel isomorphic to $\Rx$. 
We remark that the definition of $\BB$ and these two homomorphisms depend on the choice of the 
integral domain $R$ and the smooth non-trivial character $\chi$.

\begin{defin}\label{definMp}
The \emph{metaplectic group} of $W$, attached to $R$ and $\chi$, is the subgroup $\mathrm{Mp}_{R,\chi}(W)=\Sp\times_{\B} \BB$ of $\Sp\times \BB$ of the pairs $(\sigma,\mathbf{s})$ such that $\mu(\sigma)=\pi_0(\mathbf{s})$.
\end{defin}

\noindent From now on, we write $\Mp$ instead of $\mathrm{Mp}_{R,\chi}(W)$. We have a group homomorphism
\begin{equation*}
\begin{array}{rcl}
\pi:\Mp&\longrightarrow & \Sp\\
(\sigma,\mathbf{s})&\longmapsto &\sigma.
\end{array}
\end{equation*}
The morphism $\pi_0$ is surjective and surjectivity in the category of groups is preserved under base-change, therefore $\pi$ is surjective. Moreover an element $(\sigma,\mathbf{s})$ is in the kernel of $\pi$ if and only if $\mathbf{s}$ is in the kernel of $\pi_0$, that is isomorphic to $\Rx$. Thus we obtain:

\begin{teor} The following sequence is exact:
\begin{equation}\label{ses2}
1\longrightarrow \Rx \longrightarrow \Mp \stackrel{\pi}{\longrightarrow}\Sp \longrightarrow 1
\end{equation}
where $\Rx$ injects in $\Mp$ by $t\mapsto (\mathrm{id}, t	\cdot\mathrm{id}_{\S(X)})$.
\end{teor}

\noindent
Since $\BB=\B/\Rx$ and $\BB\subset \Aut(\S(X))$, we may regard $\mu$ as a projective representation of the symplectic group. Then, the metaplectic group is defined in such a way that the map
\begin{equation}\label{Weilrep}
\begin{array}{rcl}
\Mp&\longrightarrow & \BB\\
(\sigma,\mathbf{s})&\longmapsto &\mathbf{s}
\end{array}
\end{equation}
is a faithful representation on the $R$-module $\S(X)$ that lifts $\mu$.

\section{The Weil factor}

The sequence (\ref{ses2}) constitutes the object of our study and the rest of the article is devoted to study its properties.
Following the idea of Weil, we define in this section a map $\gamma$ that associates to every non-degenerate quadratic form $f$ on $X$ an invertible element $\gamma(f)\in\Rx$ (cfr. 14 of \cite{Wei}). This object, that we call \emph{Weil factor}, shows up at the moment of understanding the map $\pi$ by lifting a description of $\Sp$ by generators and relations. The study of its properties is at the heart of the results in \cite{Wei}. We prove that similar properties hold for $\gamma(f)\in\Rx$.\\ The general idea is: we find the relation (\ref{relgamma2}) in $\B$ and we lift it into $\BB$ finding an element of $\Rx$ thanks to Theorem \ref{ses1}. Then we proceed in two directions: on one hand we prove results that are useful to calculate $\gamma(f)$ while on the other we use the Weil factor to lift to $\Mp$ the relations of Proposition \ref{lemma6}.




\subsection{The Weil factor}
\noindent
Let $f\in Q^{nd}(X)$ be a non-degenerate quadratic form on $X$ and let $\rho\in \mathrm{Iso}(X, X^*)$ be its associated symmetric isomorphism. 
Explicit calculations in $\Sp$ give the equality
\begin{equation}\label{relgamma1} 
d'(\rho^{-1})t(f)d'(-\rho^{-1})t(f)=t(-f)d'(\rho^{-1}).
\end{equation}
Moreover, applying Proposition \ref{relazSp}, (\ref{relgamma1}) is equivalent to $\left(t(f)d'(\rho^{-1})\right)^3=\left(d'(\rho^{-1})t(f)\right)^3=1$.
It follows from equation (\ref{relgamma1}) that
\begin{equation}\label{relgamma2} 
d'_0(\rho^{-1})t_0(f)d'_0(-\rho^{-1})t_0(f)=t_0(-f)d'_0(\rho^{-1}).
\end{equation}

\noindent
We denote 
$\s=\s(f)=\d'_0(\rho^{-1})\t_0(f)\d'_0(-\rho^{-1})\t_0(f)$ and $\s'=\s'(f)=\t_0(-f)\d'_0(\rho^{-1})$.
We have by Proposition \ref{rialz} and equation (\ref{relgamma2}), $\pi_0(\s)=\pi_0(\s')$. Hence $\s$ and $\s'$ differ by an element of $\Rx$

\begin{defin}\label{defgamma}
Let $\gamma(f)\in\Rx$ be such that $\s=\gamma(f)\s'$. We call $\gamma(f)$ the \textit{Weil factor} associated to $f\in Q^{nd}(X)$. 
\end{defin}

\noindent
By formulas (\ref{relazBB}) we have $\gamma(f)=\left(\t_0(f)\d'_0(\rho^{-1})\right)^3=\left(\d'_0(\rho^{-1})\t_0(f)\right)^3$.

\bigskip

\noindent
We are now ready to investigate some properties of $\gamma$, starting from seeing what changes under the action of $\Aut(X)$.

\begin{prop}\label{propgamma1} Let $f\in Q^{nd}(X)$.
\begin{enumerate}
	\item[(i)] We have $\gamma(-f)=\gamma(f)^{-1}$.
	\item[(ii)] For every $\alpha\in\Aut(X)$ we have $\gamma(f^{\alpha})=\gamma(f)$.
\end{enumerate}
\end{prop}

\begin{proof}
Let $f\in Q^{nd}(X)$ be associated to the symmetric isomorphism $\rho$.
\begin{enumerate}
		\item[(i)] 
		We have
$\gamma(-f)=\left(\t_0(-f)\d'_0(-\rho^{-1})\right)^3=\left(\d'_0(\rho^{-1})\t_0(f)\right)^{-3}=\gamma(f)^{-1}.$
		\item[(ii)] The symmetric isomorphism associated to $f^{\alpha}$ is $\alpha^{*}\rho\alpha$. Then 
		we have
\begin{align*}
\gamma(f^{\alpha})&=\left(\t_0(f^{\alpha})\d'_0(\alpha^{-1}\rho^{-1}\alpha^{*-1})\right)^3
												=\left(\d_0(\alpha)^{-1}\t_0(f)\d_0(\alpha)\d_0(\alpha)^{-1}\d'_0(\rho^{-1})\d_0(\alpha)\right)^3\\
												&=\d_0(\alpha)^{-1}\left(\t_0(f)\d'_0(\rho^{-1})\right)^3\d_0(\alpha)=\gamma(f).  &\qedhere
\end{align*}
\end{enumerate}
\end{proof}

\noindent
Proposition \ref{propgamma1} gives actually a strong result in a particular case: if $-1\in (F^{\times})^2$ and $a^2=-1$ with $a\in F^{\times}$ then $x\mapsto a x$ is an automorphism of $X$. By Proposition \ref{propgamma1} we have $\gamma(f)=\gamma(-f)=\gamma(f)^{-1}$, in other words $\gamma(f)^2=1$. This does not hold in general for a local field $F$ without square roots of $-1$. 

\vspace{0.2cm}
\noindent
Let $f\in Q^{nd}(X)$ be associated to $\rho$ and define $\varphi=\chi\circ f$. Notice that $\varphi(-x)=\varphi(x)$. 
For every $\Phi\in\S(X)$, we denote by $\Phi*\varphi$ the convolution product defined by 
\[(\Phi*\varphi)(x)=\int_X \Phi(x')\varphi(x-x')dx'\] for every $x\in X$. We have that $ \Phi*\varphi\in\S(X)$, indeed 
\begin{align*}
(\Phi*\varphi)(x)&=\int_X \Phi(x')\varphi(x-x')dx'=\varphi(x)\int_X \Phi(x')\varphi(-x')\la x,\rho(-x')\ra dx'\\
&=\varphi(x)\int_X \Phi(x')\varphi(x')\la x',-\rho(x)\ra dx'=|\rho|^{-\frac{1}{2}}\t_0(f)\d_0(-\rho^{-1})\t_0(f)\Phi(x)
\end{align*}
where we have used that $\varphi(x+y)=\varphi(x)\varphi(y)\la x,\rho(y)\ra$ for every $x,y\in X$.

\vspace{0.25cm}
\noindent Now we state a proposition that gives a summation formula for $\gamma(f)$ and that allows us to calculate in Theorem \ref{gamma-1} the value of $\gamma$ for a specific quadratic form over $F$.

\begin{prop}\label{calgamma} Let $f\in Q^{nd}(X)$ be associated to the symmetric isomorphism $\rho\in \Isom(X,X^*)$ and let $\s, \s'\in \BB$ as in Definition \ref{defgamma}. We set $\varphi=\chi\circ f$. 

\begin{enumerate}
	\item For every $\Phi\in\S(X)$ and for every $x\in X$ we have 
$$\s\Phi(x)=|\rho| \Fou(\Phi*\varphi)(\rho(x)) \quad\text{ and }\quad \s'\Phi(x)=|\rho|^{\frac{1}{2}}\Fou\Phi(\rho(x)) \varphi(x)^{-1}.$$
	\item For every $\Phi\in\S(X)$ and for every $x^*\in X^*$ we have
\begin{equation}\label{calgamma1}
\Fou(\Phi*\varphi)(x^*)=\gamma(f)|\rho|^{-\frac{1}{2}}\Fou\Phi(x^*)\varphi(\rho^{-1}x^*)^{-1}.
\end{equation}
	\item There exists a sufficiently large compact open subgroup $K_0$ of $X$ such that for every compact open subgroup $K$ of $X$ containing $K_0$ and for every $x^*\in X^*$, the integral $\int_K\varphi(x)\la x,x^*\ra dx$ does not depend on $K$. Moreover we have
\begin{equation}\label{calgamma2}
\int_K\varphi(x)\la x,x^*\ra dx=\gamma(f)|\rho|^{-\frac{1}{2}}\varphi(\rho^{-1}x^*)^{-1}
\end{equation}
and we denote $\Fou\varphi=\int_K\varphi(x)\la x,x^*\ra dx$.
	\item If $K$ is a sufficiently large compact open subgroup of $X$, we have
\begin{equation}\label{calgamma3}
\gamma(f)=|\rho|^{\frac{1}{2}}\int_K \chi(f(x))dx.
\end{equation}
\end{enumerate}
\end{prop}

\begin{proof} \mbox{}
\begin{enumerate}
		\item For every $\Phi\in\S(X)$ and every $x\in X$ we have
\begin{align*}
\s\Phi(x)&=\d'_0(\rho^{-1})\t_0(f)\d'_0(-\rho^{-1})\t_0(f)\Phi(x)\\
					 &=|\rho|\int_X\int_X\Phi(x_1)\varphi(x_1)\la x_1,-\rho(x_2)\ra \varphi(x_2)\la x_2,\rho(x)\ra dx_1 dx_2\\
					 &=|\rho|\int_X\int_X\Phi(x_1)\varphi(-x_1)\la x_1,-\rho(x_2)\ra \varphi(x_2)\la x_2,\rho(x)\ra dx_1 dx_2\\
					 &=|\rho|\int_X\int_X\Phi(x_1)\varphi(x_2-x_1)\la x_2,\rho(x)\ra dx_1 dx_2
	         =|\rho| \Fou(\Phi*\varphi)(\rho(x))
\end{align*}
and $\s'\Phi(x)=\t_0(-f)\d'_0(\rho^{-1})\Phi(x)=\t_0(-f)|\rho|^{\frac{1}{2}}\Fou(\Phi\circ\rho)(x)=\varphi(x)^{-1}|\rho|^{\frac{1}{2}}\Fou(\Phi\circ\rho)(x)$.
		\item By the equality $\s=\gamma(f)\s'$ we have $|\rho| \Fou(\Phi*\varphi)(\rho(x))=\gamma(f)|\rho|^{\frac{1}{2}}\Fou\Phi(\rho(x))\varphi(x)^{-1}$ and replacing $\rho(x)$ by $x^*$ we obtain the equality (\ref{calgamma1}).
		\item Taking $\Phi=\I_H$ for a compact open subgroup $H$ of $X$ in formula (\ref{calgamma1}), we obtain
$$\int_X (\I_H*\varphi)(x_1)\la x_1,x^*\ra dx_1=\gamma(f)|\rho|^{-\frac{1}{2}} \Fou\I_{H}(x^*)\varphi(\rho^{-1}x^*)^{-1}.$$
We want to calculate the integral in the left hand side. We can take a compact open subgroup $K_0$ of $X$ large enough to contain both $H$ and the support of $\I_H*\varphi$ obtaining
$$\int_X (\I_H*\varphi)(x_1)\la x_1,x^*\ra dx_1=\int_{K_0}\int_H\varphi_{|_{K_0}}(x_1-x_2)dx_2\la x_1,x^*\ra dx_1.$$
Now, we can prove that $\varphi_{|_{K_0}}$ is locally constant and that we can change the order of the two integrals, i.e.
\begin{align*}
\int_X (\I_H*\varphi)(x_1)\la x_1,x^*\ra dx_1&=\int_H\int_{K_0}\varphi_{|_{K_0}}(x_1-x_2)\la x_1,x^*\ra dx_1dx_2\\
																						 &=\int_H\int_{K_0}\varphi_{|_{K_0}}(x'_1)\la x'_1+x_2,x^*\ra dx'_1dx_2\\
																						 &=\Fou\I_H(x^*)\int_{K_0}\varphi_{|_{K_0}}(x'_1)\la x'_1,x^*\ra dx'_1.
\end{align*}
Since $\Fou\I_H=\mathrm{vol}(H)\I_{H_*}$ and $\mathrm{vol}(H)\neq 0$, we obtain the equality (\ref{calgamma2}) for every $x^*\in H_*$ and every $H$ compact open subgroup of $X$. Now $H_*$ cover $X^*$, varying $H$, and so the equality holds for every $x^*\in X^*$ . It is clear that the equality holds also for every compact open subgroup $K$ of $X$ containing $K_0$.
\item Setting $x^*=0$ in (\ref{calgamma2}) we obtain
$\gamma(f)=|\rho|^{\frac{1}{2}}\int_K \varphi(x)dx=|\rho|^{\frac{1}{2}}\int_K \chi(f(x))dx.$\qedhere
\end{enumerate}
\end{proof}

\begin{rmk}\label{propgamma3}
The second result in Proposition \ref{propgamma1} is true more generally for every $\alpha'\in \Isom(X', X)$ where $X'$ is a finite dimensional $F$-vector space. In fact if $K'$ is a compact open subgroup of $X'$ large enough,  $f\in Q^{nd}(X)$ and $\alpha\in\Isom(X',X)$ by (\ref{calgamma3}) we have
\begin{align*}
\gamma(f\circ\alpha)&=|\alpha^*\rho\alpha|^{\frac{1}{2}}\int_{K'} \chi(f(\alpha(x')))dx'=|\rho|^{\frac{1}{2}}|\alpha|\int_{X'}\I_{\alpha(K')}(\alpha(x'))\chi(f(\alpha(x')))dx'\\
&=|\rho|^{\frac{1}{2}}\int_{X}\I_{\alpha(K')}(x)\chi(f(x))dx'=\gamma(f).
\end{align*}
\end{rmk}


\subsubsection{Symplectic generators in $\BB$}

\begin{defin}\label{defr0}
Let $\sigma\in\Omega(W)$. By Proposition \ref{omegaSp} we can write $\sigma=t(f_1)d'(\beta)t(f_2)$ for unique $f_1,f_2\in Q(X)$ and $\beta\in\Isom(X^*,X)$. We define a map $\r_0:\Omega(W)\to\BB$ by $$\r_0(\sigma)=\t_0(f_1)\d'_0(\beta)\t_0(f_2)$$
for every $\sigma\in\Omega(W)$.
\end{defin}

\noindent
Now we state a theorem that says how an equality $\sigma''=\sigma\sigma'$ in $\Omega(W)$ lifts to $\BB$. After a comparison with section 15 of \cite{Wei} the differences turn out to be the use of Fourier transform for Schwartz functions and previous changes in notations. Finally we have clarified some points and made them explicit.

\begin{teor}\label{teoremar}
Let $\sigma=\begin{pmatrix}\alpha&\beta\\ \gamma&\delta\end{pmatrix}$, $\sigma'=\begin{pmatrix}\alpha'&\beta'\\ \gamma'&\delta'\end{pmatrix}$ and $\sigma''=\begin{pmatrix}\alpha''&\beta''\\ \gamma''&\delta''\end{pmatrix}$ be elements of $\Omega(W)$ such that $\sigma''=\sigma\sigma'$. Then 
$$\r_0(\sigma)\r_0(\sigma')=\gamma(f_0)\r_0(\sigma'')$$
where $f_0$ is the non-degenerate quadratic form on $X$ associated to the symmetric isomorphism $-\beta^{-1}\beta''\beta'^{-1}:X\rightarrow X^*$.
\end{teor}

\begin{proof}
Since $\r_0(\sigma)\r_0(\sigma')$ and $\r_0(\sigma'')$ have the same image by $\pi_0$, we can set $\r_0(\sigma)\r_0(\sigma')=\lambda\r_0(\sigma'')$ where $\lambda\in\Rx$ depends on $\sigma,\sigma'$. By Definition \ref{defr0} we have
$$\t_0(f_1)\d'_0(\beta)\t_0(f_2)\t_0(f'_1)\d'_0(\beta')\t_0(f'_2)=\lambda\t_0(f''_1)\d'_0(\beta'')\t_0(f''_2)$$
for suitable $f_1, f_2, f'_1, f'_2, f''_1, f''_2 \in Q(X)$. Setting $f_0=f_2+f'_1$, $f_3=-f_1+f''_1$ and $f_4=f''_2-f'_2$ we obtain
$$\d'_0(\beta)\t_0(f_0)\d'_0(\beta')=\d'_0(\beta)\t_0(f_0)\d'_0(-\beta'^*)^{-1}=\lambda\t_0(f_3)\d'_0(\beta'')\t_0(f_4)$$
where we have used that $\d'_0(\beta')^{-1}=\d'_0(-\beta'^*)$. By Remark \ref{dimOmegaSp} the symmetric homomorphisms associated to $f_2$ and $f'_1$ are $\rho_2=-\beta^{-1}\alpha$ and $\rho'_1=-\delta'\beta'^{-1}$, hence the symmetric homomorphism associated to $f_0$ is $\rho_0=\rho_2+\rho'_1=-\beta^{-1}(\alpha\beta'+\beta\delta')\beta'^{-1}=-\beta^{-1}\beta''\beta'^{-1}=-\beta'^{*-1}\beta''^*\beta^{*-1}$ that is also an isomorphism.\\
We set $\varphi_i=\chi\circ f_i$ for $i=0,3,4$. For every $\Phi\in\S(X)$ and $x\in X$ we have
$$\d'_0(\beta)\t_0(f_0)\d'_0(-\beta'^{*})^{-1}\Phi(x)=|\beta|^{-\frac{1}{2}}|\beta'|^{\frac{1}{2}}\Fou(\Fou^{-1}(\Phi\circ(-\beta'^{*})) \cdot \varphi_0)(\beta^{-1}x).$$
By Proposition \ref{teorFourier} the Fourier transform of a pointwise product is the convolution product of the Fourier transforms and then $\d'_0(\beta)\t_0(f_0)\d'_0(\beta')\Phi(x)=|\beta|^{-\frac{1}{2}}|\beta'|^{\frac{1}{2}}\big((\Phi\circ\beta'^{*})*\Fou\varphi_0\big)(\beta^{-1}x).$
Using formula (\ref{calgamma2}) we obtain
\begin{align*}
\d'_0(\beta)\t_0(f_0)\d'_0(\beta')\Phi(x)&=\gamma(f_0)|\rho_0|^{-\frac{1}{2}}|\beta|^{-\frac{1}{2}}|\beta'|^{\frac{1}{2}}\big((\Phi\circ\beta'^{*})*(\varphi_0\circ\rho_0^{-1})^{-1}\big)(\beta^{-1}x)\\
&=\gamma(f_0)|\beta''|^{-\frac{1}{2}}|\beta'|\big((\Phi\circ\beta'^{*})*(\varphi_0\circ\rho_0^{-1})^{-1}\big)(\beta^{-1}x)\\
&=\gamma(f_0)|\beta''|^{-\frac{1}{2}}|\beta'| \int_{X^*}\Phi(\beta'^*(x^*)) \varphi_0(\beta^*\beta''^{*-1}\beta'^*(x^*)-\beta'\beta''^{-1}(x))^{-1}dx^*\\
&=\gamma(f_0)|\beta''|^{-\frac{1}{2}}\int_X \Phi(x_1) \varphi_0(-\beta'\beta''^{-1}(x)+\beta^*\beta''^{*-1}(x_1))^{-1}dx_1
\end{align*}
where in the last step we have used the change of variables $\beta'^*(x^*)\mapsto x_1$.
Furthermore we have
$$\t_0(f_3)\d'_0(\beta'')\t_0(f_4)\Phi(x)=|\beta''|^{-\frac{1}{2}}\int_X \Phi(x_1) \varphi_4(x_1)\varphi_3(x)\la x_1,\beta''^{-1}x \ra dx_1$$
and then 
$$\gamma(f_0)\int_X \Phi(x_1) \varphi_0(-\beta'\beta''^{-1}(x)+\beta^*\beta''^{*-1}(x_1))^{-1}dx_1=\lambda\int_X \Phi(x_1) \varphi_4(x_1)\varphi_3(x)\la x_1,\beta''^{-1}x \ra dx_1.$$
We observe that the two sides are of the form $c_i\int_X\Phi(x_1)\vartheta_i(x_1,x) dx_1$ for $i=1,2$, where $c_i\in \Rx$ and $\vartheta_i$ are characters of degree 2 of $X\times X$. Since the equality holds for every $\Phi\in\S(X)$ and every $x\in X$, we obtain that $c_1=c_2$ and $\vartheta_1=\vartheta_2$ and so $\gamma(f_0)=\lambda$.
\end{proof}


\subsection{Metaplectic realizations of forms}\label{applMp}
Definitions \ref{applSp} and \ref{applBB} allow us to define some applications from $\Aut(X)$, $\Isom(X^*,X)$ and $Q(X)$ to $\Mp$, similar to those in 34 of \cite{Wei}, that satisfy relations analogous to those of $\d_0$, $\d'_0$ and $\t_0$.

\begin{defin}\label{defapplMp}
Let $\Mp$ be as in Definition \ref{definMp}. We define the following applications.
\begin{itemize}
		\item The injective group homomorphism $\d:\Aut(X)\longrightarrow \Mp$ given by $\d(\alpha)=(d(\alpha),\d_0(\alpha))$ for every $\alpha\in\Aut(X)$.
		\item The injective map $\d':\Isom(X^*,X)\longrightarrow \Mp$ given by $\d'(\beta)=(d'(\beta),\d'_0(\beta))$ for every $\beta\in \Isom(X^*,X)$.
		\item The injective group homomorphism $\t:Q(X)\longrightarrow \Mp$ given by $\t(f)=(t(f),\t_0(f))$ for every $f\in Q(X)$.
\end{itemize}
\end{defin}

\noindent
By Proposition \ref{relazSp} and by (\ref{relazBB}) we have 
\begin{equation}\label{relazMp}
\d(\alpha)^{-1}\t(f)\d(\alpha)=\t(f^{\alpha})
\end{equation} 
for every $f\in Q(X)$ and $\alpha\in \Aut(X)$. We have also
$\d'(\alpha\circ\beta)=\d(\alpha)\d'(\beta)$ and $\d'(\beta\circ\alpha^{*-1})=\d'(\beta)\d(\alpha)$
for every $\alpha\in \Aut(X)$ and $\beta\in \Isom(X^*,X)$.

\vspace{0.25cm}
\noindent
As in Definition \ref{defr0}, we can define a map from $\Omega(W)$ to $\Mp$.  By Proposition \ref{omegaSp} every element $\sigma\in\Omega(W)$ can be written uniquely as $\sigma=t(f_1)d'(\beta)t(f_2)$: we define
\begin{equation}\label{definrMp}
\r(\sigma)=\t(f_1)\d'(\beta)\t(f_2)
\end{equation}
that is equivalent to write $\r(\sigma)=(\sigma,\r_0(\sigma))$.

\bigskip
\noindent
Let $\sigma=\begin{pmatrix}\alpha&\beta\\ \gamma&\delta\end{pmatrix}$, $\sigma'=\begin{pmatrix}\alpha'&\beta'\\ \gamma'&\delta'\end{pmatrix}$ and $\sigma''=\begin{pmatrix}\alpha''&\beta''\\ \gamma''&\delta''\end{pmatrix}$ be in $\Omega(W)$ such that $\sigma\sigma'=\sigma''$. By Theorem \ref{teoremar} we have
\begin{equation}\label{formular}
\r(\sigma)\r(\sigma')=\gamma(f_0)\r(\sigma'')
\end{equation}
where $f_0$ is the non-degenerate quadratic form on $X$ associated to the symmetric isomorphism $-\beta^{-1}\beta''\beta'^{-1}$. 


\section{Fundamental properties of the Weil factor}
\noindent In this section we find the possible values of $\gamma(f)$ for every non-degenerate quadratic form $f$ over $F$. Proposition \ref{calgamma} gives a summation formula for $\gamma(f)$ and we use it to prove that $\gamma(n)=-1$ where $n$ is the reduced norm of the quaternion division algebra over $F$. In Theorem \ref{Witt} we see that $\gamma$ is a $R$-character of the Witt group of $F$. Moreover we already know by Proposition \ref{propgamma1} that $\gamma(f)^2=1$ if $F$ contains a square root of $-1$ and at the end of this section this is generalized by saying that, for any $F$, $\gamma(f)$ is a fourth root of unity in $R$.

\smallskip
\noindent For every positive integer $m$, we denote by $q_m$ the non-degenerate quadratic form $q_m(x)=\sum_{i=1}^m x_i^2$ defined on the $m$-dimensional vector space $F^m$.


\subsection{The quaternion division algebra over $F$}
In this paragraph we use some results on quaternion algebras over $F$ (\cite{Vig2}) to prove that if $\mathrm{char}(R)\neq 2$ the map $\gamma:Q^{nd}(X)\longrightarrow \Rx$ is non-trivial by means of a concrete example.

\bigskip

\noindent Let $Y,Z,T$ be three pro-$p$ groups such that there exists an exact sequence of continuous maps $1\longrightarrow Y \stackrel{i}{\longrightarrow} Z \stackrel{j}{\longrightarrow} T \longrightarrow 1$.
We say that the Haar measures $dy,dz,dt$ on $Y,Z,T$ with values in $R$ are \emph{compatible} with the exact sequence above (II.4 of \cite{Vig2}) if for every locally constant function $\vartheta:Z\longrightarrow R$ 
we have the equality
$$\int_Z\vartheta(z)dz=\int_T\left(\int_Y\vartheta(i(y)z)dy\right)dt.$$
The function $\int_Y\vartheta(i(y)z)dy$ is invariant by $z\mapsto i(y)z$ for every $y\in Y$ and so we can see it as a function in the variable $t$.

\bigskip
\noindent
By Theorem II.1.1 of \cite{Vig2} we know that there exists a unique quaternion division algebra over $F$ (up to isomorphism) that we denote by $A$. 
The reduced norm $n:A\longrightarrow F$ is a non-degenerate quadratic form on the $F$-vector space underlying $A$ and it induces a surjective group homomorphism $n_{|A^\times}:A^{\times}\longrightarrow F^{\times}$. Moreover by Lemma II.1.4 of \cite{Vig2}, if $v$ is a discrete valuation of $F$ such that $v(\varpi)=1$ then $v\circ n$ is a discrete valuation of $A$; so we can consider the ring of integers $\mathcal{O}_A=\{z\in A\,|\, n(z)\in\mathcal{O}_F\}$ of $A$ and fix a uniformizer $\varpi_A$ of $\mathcal{O}_A$ such that $\varpi_A^2=\varpi$. The unique prime ideal of $\mathcal{O}_A$ is $\varpi_A\mathcal{O}_A$ and the cardinality of the residue field of $A$ is $q^2$ where $q$ is the cardinality of the residue field of $F$.

\bigskip
\noindent According to Definition \ref{module}, we define the module of $x\in F$ (resp. $z\in A$), denoted by $|x|$ (resp. $|z|_A$), as the module of the multiplication (resp. right multiplication) by $x$ (resp. $z$). We can easily prove that $|x|=q^{-v(x)}$ and $|z|_A=|n(z)|^{2}$.

\bigskip
\noindent
We denote by $dx$ and $dz$ the Haar measures on $F$ and $A$ such that $\mathrm{vol}(\OF,dx)=\mathrm{vol}(\mathcal{O}_A,dz)=1$ and by $dx^{\times}=|x|^{-1}dx$ and $dz^{\times}=|n(z)|^{-2}dz$ the Haar measures on $F^{\times}$ and $A^{\times}$.\\
It is easy to see that $n_{|\mathcal{O}_A^{\times}}:\mathcal{O}_A^{\times}\longrightarrow \mathcal{O}_F^{\times}$ and $n_{|U^1_A}:U^1_A=1+\varpi_A\mathcal{O}_A\longrightarrow U^1_F=1+\varpi\mathcal{O}_F$ are two surjective group homomorphisms.
We observe that $\mathrm{vol}(U^1_A,dz^{\times})=\mathrm{vol}(U^1_A,dz)=\mathrm{vol}(\varpi_A\mathcal{O}_A,dz)=q^{-2}$ and $\mathrm{vol}(U^1_F,dx^{\times})=q^{-1}$. Now, if we take the Haar measure $dy$ on the kernel of $n_{|U^1_A}$ such that $\mathrm{vol}(\ker(n_{|U^1_A}),dy)=q^{-1}$, the Haar measures $dy, dz^{\times}, dx^{\times}$ are compatible with the exact sequence $1\longrightarrow \ker(n_{|U^1_A}) \longrightarrow U^1_A \stackrel{n}{\longrightarrow} U^1_F \longrightarrow 1$. Indeed, they are compatible if for every compact open subgroup $K$ of $U^1_A$ we have
\begin{equation}\label{compatibilità}
\mathrm{vol}(K,dz^{\times})=\int_{U^1_F}\bigg(\int_{\ker(n_{|U^1_A})}\I_K(yz)dy\bigg)dx^{\times}=\int_{U^1_F}\mathrm{vol}(Kz^{-1}\cap \ker(n),dy)dx^{\times}
\end{equation}
where the function $z\mapsto \mathrm{vol}(Kz^{-1}\cap \ker(n),dy)$ is invariant by $z\mapsto yz$ for every $y\in \ker(n_{|U^1_A})$ and so we can see it as a function in the variable $x\in U^1_F$. Now, if $n(z)\notin n(K)$ then $Kz^{-1}\cap \ker(n)=\emptyset$ while if $n(z)=n(k)$ with $k\in K$ then $Kz^{-1}\cap \ker(n)=(K\cap n^{-1}(n(z)))z^{-1}=n_{|K}^{-1}(n_{|K}^{}(k))z^{-1}=(K\cap\ker(n))k z^{-1}$. Then (\ref{compatibilità}) becomes $\mathrm{vol}(K,dz^{\times})=\mathrm{vol}(K\cap\ker(n),dy)\cdot\mathrm{vol}(n(K),dx^{\times})$ or equivalently $q^{-2}[U^1_A:K]=q^{-1}[\ker(n):K\cap\ker(n)]\cdot q^{-1}[U^1_F:n(K)]$ that is clearly true.

\begin{teor}\label{gamma-1}
Let $A$ be the quaternion division algebra over $F$ and let $n:A\longrightarrow F$ be the reduced norm of $A$. Then $\gamma(n)=-1$.
\end{teor}

\begin{proof} 
Let $\rho_n\in\Isom(A,A^*)$ be the symmetric isomorphism associated to the quadratic form $n$. For every $\lambda\in\Z$, let 
$M_{\lambda}=\varpi_A^{-\lambda}\mathcal{O}_A=\{z\in A\,|\,n(z)\in \varpi^{-\lambda}\OF\}$. 
By (\ref{calgamma3}) we have $$\gamma(n)=|\rho_n|^{\frac{1}{2}}\int_{M_{\lambda}}\chi(n(z))dz$$
for $\lambda\geq 0$ large enough. 
Let $l$ be the conductor of $\chi$; then $\int_{M_{-l}}\chi(n(z))dz=\mathrm{vol}(M_{-l})=\mathrm{vol}(\varpi_A^l\mathcal{O}_A)=q^{-2l}$. Hence, if we choose $\lambda\geq 1-l$, we obtain
\begin{align*}
\int_{M_{\lambda}}\chi(n(z))dz&=\mathrm{vol}(M_{-l})+\int_{M_{\lambda}- M_{-l}}\chi(n(z))dz
															=q^{-2l}+\sum_{i=-\lambda}^{l-1}\int_{\varpi_A^{i}\mathcal{O}_A^{\times}}\chi(n(z))dz\\
															&=q^{-2l}+\sum_{i=-\lambda}^{l-1}\int_{\varpi_A^{i}\mathcal{O}_A^{\times}}\chi(n(z))|n(z)|^2dz^{\times}
															=q^{-2l}+\sum_{i=-\lambda}^{l-1}q^{-2i}\int_{\varpi_A^{i}\mathcal{O}_A^{\times}}\chi(n(z))dz^{\times}.
\end{align*}
Now we fix a set of representatives $\Xi_A$ of $\mathcal{O}_A^{\times}/U^1_A$. Then, for every $-\lambda\leq i\leq l-1$, we have
$$\int_{\varpi_A^{i}\mathcal{O}_A^{\times}}\chi(n(z))dz^{\times}=\sum_{\xi\in\Xi_A}\int_{U^1_A}\chi(n(\varpi_A^{i}\xi z))dz^{\times}=\sum_{\xi\in\Xi_A}\int_{U^1_A}\chi((-\varpi)^{i}n(\xi)n(z))dz^{\times}.$$
We already know that $\mathrm{vol}(\ker(n_{|U^1_A}),dy)=q^{-1}$, so using the compatibility of $dy, dz^{\times}, dx^{\times}$ we obtain 
$$\int_{\varpi_A^{i}\mathcal{O}_A^{\times}}\chi(n(z))dz^{\times}=\sum_{\xi\in\Xi_A}q^{-1}\int_{U^1_F}\chi((-\varpi)^{i}n(\xi)x)dx^{\times}.$$
The morphisms $n_{|\mathcal{O}_A^{\times}}$ gives a surjective morphism $\mathcal{O}_A^{\times}/U^1_A\longrightarrow \mathcal{O}_F^{\times}/U^1_F$ which kernel is of order $q+1$. Then, if we fix a set of representatives $\Xi_F$ of $\mathcal{O}_F^{\times}/U^1_F$ we obtain
$$\int_{\varpi_A^{i}\mathcal{O}_A^{\times}}\chi(n(z))dz^{\times}=q^{-1}(q+1)\sum_{\xi'\in\Xi_F}\int_{U^1_F}\chi((-\varpi)^{i}\xi'x)dx^{\times}=
q^{-1}(q+1)\int_{\varpi^i\mathcal{O}_F^{\times}}\chi(x)dx^{\times}$$
and hence
\begin{align*}
\int_{M_{\lambda}-M_{-l}}\chi(n(z))dz&=q^{-1}(q+1)\sum_{i=-\lambda}^{l-1}q^{-2i}\int_{\varpi^{i}\mathcal{O}_F^{\times}}\chi(x)dx^{\times}=q^{-1}(q+1)\sum_{i=-\lambda}^{l-1}q^{-i}\int_{\varpi^i\mathcal{O}_F^{\times}}\chi(x)dx\\																		&=q^{-1}(q+1)\sum_{i=-\lambda}^{l-1}q^{-i}\left(\int_{\varpi^i\mathcal{O}_F}\chi(x)dx-\int_{\varpi^{i+1}\mathcal{O}_F}\chi(x)dx\right).
\end{align*}
We observe that we have $\int_{\varpi^{j}\OF}\chi(x)dx= q^{-j}$ if $j\geq l$ and $0$ otherwise. Indeed if $j<l$, then there exists $x'\in\varpi^{j}\OF$ such that $\chi(x')\neq 1$; hence $\int_{\varpi^{j}\OF}\chi(x)dx=\int_{\varpi^{j}\OF}\chi(x+x')dx=\chi(x')\int_{\varpi^{j}\OF}\chi(x)dx$ that implies $\int_{\varpi^{j}\OF}\chi(x)dx=0$. Then
$$\int_{M_{\lambda}}\chi(n(z))dz=q^{-2l}+q^{-1}(q+1)q^{-l+1}(-q^{-l})=q^{1-2l}.$$
Finally we obtain $\gamma(n)=-|\rho_n|^{\frac{1}{2}}q^{1-2l}$.  By Definition \ref{module}, where we set $\Phi=\I_{(\mathcal{O}_A)_*}$,  we have $$|\rho_n|=\mathrm{vol}(\rho^{-1}((\mathcal{O}_A)_*),dz)^{-1}.$$
Moreover it is easy to show that $\rho_n(z_1)(z_2)=tr(z_1\bar{z_2})$ for every $z_1,z_2\in A$, where $z\mapsto \bar{z}$ is the conjugation of $A$ (cf. page 1 of \cite{Vig2}). Then the following are equivalent
$$z\in \rho^{-1}((\mathcal{O}_A)_*) \Longleftrightarrow \la z,\rho(\mathcal{O}_A)\ra=1 \Longleftrightarrow tr(z\mathcal{O}_A)\subset\ker(\chi).$$
We know that $\{z\in A\,|\,tr(z\mathcal{O}_A)\subset\OF\}$ is a fractional ideal (its inverse is called codifferent ideal), and by Corollary II.1.7 of \cite{Vig2} it is exactly $\varpi_A^{-1}\mathcal{O}_A$. Then $z\in \rho^{-1}((\mathcal{O}_A)_*)$ if and only if $z\in \varpi^l\varpi_A^{-1}\mathcal{O}_A=\varpi_A^{2l-1}\mathcal{O}_A$. Hence $|\rho_n|=q^{4l-2}$ and $\gamma(n)=-1$.
\end{proof}
\begin{rmk}
The Theorem \ref{gamma-1} corresponds to Proposition 4 of \cite{Wei}. Weil proves it showing that $\gamma(n)$ is a negative real number of absolute value 1 and hence we need further remarks to conclude the proof. Calculating explicitly the volume $\mathrm{vol}(\rho^{-1}((\mathcal{O}_A)_*),dz)$ we could conclude for $R$ of characteristic zero.\\ In the general case we are no more sure that we are calculating integrals over subsets of non-zero volume so we can solve the problem making integrals over pro-$p$-groups.
\end{rmk}


\subsection{The Witt group}
In this paragraph we introduce the definition of Witt group of $F$ and we prove that $\gamma$ defines a $R$-character of this group.

\bigskip

\noindent 
Let $G_1, G_2$ be two finite dimensional vector spaces over $F$ and $f_1, f_2$ be two non-degenerate quadratic forms on $G_1$ and $G_2$. 
We define $f_1\oplus f_2 \in Q^{nd}(G_1\times G_2)$ by $(f_1\oplus f_2)(x_1\oplus x_2)=f_1(x_1)+f_2(x_2)$ for every $x_1\in G_1$ and $x_2\in G_2$.
 
\begin{rmk}
If $\rho_1:G_1\to G_1^*$ and $\rho_2: G_2\to G_2^*$ are the symmetric isomorphisms associated to $f_1$ and $f_2$, then $\rho_1\oplus\rho_2:G_1\times G_2\to (G_1\times G_2)^*$, defined by $(\rho_1\oplus\rho_2)(y_1 \oplus y_2)=\rho_1(y_1) \oplus \rho_2(y_2)$ is the symmetric isomorphism associated to $f_1\oplus f_2$. Indeed, calling this latter $\rho_{_{1,2}}$, we have $$[x_1\oplus x_2,(\rho_1\oplus\rho_2)(y_1\oplus y_2)]
= f_1(x_1+y_1)-f_1(x_1)-f_1(y_1)+f_2(x_2+y_2)-f_2(x_2)-f_2(y_2)=$$ $$=(f_1\oplus f_2)(x_1\oplus x_2+y_1\oplus y_2)-(f_1\oplus f_2)(x_1\oplus x_2)-(f_1\oplus f_2)(y_1\oplus y_2)=[x_1\oplus x_2,\rho_{_{1,2}}(y_1\oplus y_2)].$$
\end{rmk}
\begin{defin}
We say that $f_1\in Q^{nd}(G_1)$ and $f_2\in Q^{nd}(G_2)$ are \emph{equivalent} (and we write $f_1\sim f_2$) if one can be obtained from the other by adding an hyperbolic quadratic form of dimension $\max\{\dim(G_1), \dim(G_2)\}-\min\{\dim(G_1), \dim(G_2)\}$ (see \cite{MiHu}). We call \emph{Witt group} of $F$ the set of equivalence classes of non-degenerate quadratic forms over $F$ endowed with the operation induced by $(f_1,f_2)\longmapsto f_1\oplus f_2$.
\end{defin}

\begin{teor}\label{Witt}
The map $f \mapsto \gamma(f)$ is a $R$-character of the Witt group of F.
\end{teor}

\begin{proof}
Let $G_1$ and $G_2$ be two finitely dimensional vector spaces over $F$, $f_1 \in Q^{nd}(G_1)$ and $f_2\in Q^{nd}(G_2)$. Proposition \ref{calgamma} gives
$$\gamma(f_1\oplus f_2)=|\rho_1 \oplus \rho_2|^{\frac{1}{2}}\int_{K_1\times K_2} \chi((f_1\oplus f_2)(x_1\oplus x_2))dx_1dx_2$$
for compact open subgroups $K_1$ and $K_2$ of $G_1$ and $G_2$, both large enough.
Now, if we consider $\I_{K_{1,*}}\in\S(G_1^*)$, $\I_{K_{2,*}}\in\S(G_2^*)$ and $\I_{K_{1,*}\times K_{2,*}}\in\S(G_1^*\times G_2^*)$, Definition \ref{module} gives
$$|\rho_1||\rho_2|\int_{G_1}\I_{K_{1,*}}(\rho_1(x_1))dx_1  \int_{G_2}\I_{K_{2,*}}(\rho_2(x_2))dx_2=\int_{G_1^*}\I_{K_{1,*}}(x_1^*)dx_1^*\int_{G_2^*}(x_2^*)\I_{K_{2,*}}dx_2^*=$$
$$=\int_{G_1^{*}\times G_2^{*}} \I_{K_{1,*}\times K_{2,*}}(x_1^*\oplus x_2^*)dx_1^* dx_2^* = |\rho_1\oplus \rho_2|\int_{G_1\times G_2} \I_{K_{1,*}\times K_{2,*}}(\rho_1(x_1)\oplus \rho_2(x_2)) dx_1dx_2$$
and then $|\rho_1||\rho_2|=|\rho_1\oplus \rho_2|$. Hence we obtain 
$$\gamma(f_1\oplus f_2)=|\rho_1|^{\frac{1}{2}}|\rho_2|^{\frac{1}{2}} \int_{K_1} \chi(f_1(x_1))dx_1\int_{K_2} \chi(f_2(x_2)) dx_2=\gamma(f_1)\gamma(f_2).$$
We shall now to check that $\gamma$ is equivariant on the equivalence classes of bilinear forms. To see that, recall that $f_1\sim f_2$ if and only if there exist $n\in \N$ and an hyperbolic quadratic form $h(\mathbf{x})=\sum x_ix_{i+n}$ of rank $2n$ such that $f_1=f_2\oplus h$. After what proven in the first part $\gamma(f_1)=\gamma(f_2)$ if and only if $\gamma(h)=1$ and since every hyperbolic form is a sum of the rank 2 form $h_2:(x_1, x_2)\mapsto x_1x_2$ it's sufficient to show that $\gamma(h_2)=1$. Now, if we apply the base change $x_1\mapsto x_1+x_2$ and $x_2\mapsto x_1-x_2$ we obtain $h_2(x_1+x_2, x_1-x_2)=(x_1+x_2)(x_1-x_2)=x_1^2-x_2^2$ and Proposition \ref{propgamma1} gives that $\gamma(h_2)=\gamma(q_1\oplus (-q_1)) = \gamma(q_1)\gamma(q_1)^{-1}=1$. 
\end{proof}


\subsection{The image of the Weil factor}
We exploit some classical results on quadratic forms over $F$ to prove that $\gamma$ takes values in the group of fourth roots of unity in $R$. 

\begin{defin}
Let $G_1,G_2$ be two finite dimensional vector spaces over $F$ and $f_1,f_2$ be two non-degenerate quadratic forms on $G_1$ and $G_2$. We say that $f_1$ and $f_2$ are \emph{isometric} if there exists an isomorphism $\vartheta:G_1\longrightarrow G_2$ such that $f_1(x)=f_2(\vartheta(x))$ for every $x\in G_1$.
\end{defin} 
\noindent
Notice that, by Remark \ref{propgamma3}, if $f_1$ and $f_2$ are isometric then $\gamma(f_1)=\gamma(f_2)$. We know also that there are only two isometry classes of non-degenerate quadratic forms on a $4$-dimensional vector space over $F$ whose discriminant is a square in $F^{\times}$. One class is represented by the norm $n$ over the quaternion division algebra and the other by $q_2\oplus -q_2$. Moreover, if $a,b\in F^{\times}$ and $\left(\frac{a}{b}\right)$ is the Hilbert symbol with values in $\Rx$, the quadratic form $x_1^2-ax_2^2-bx_3^2+abx_4^2$ lies in the first class if $\left(\frac{a}{b}\right)=-1$ and in the second one if $\left(\frac{a}{b}\right)=1$. Furthermore by Theorems \ref{Witt} and \ref{gamma-1} we have that 
\begin{equation}\label{formulagamma0}
\gamma(x_1^2-ax_2^2-bx_3^2+abx_4^2)=\left(\frac{a}{b}\right).
\end{equation}
In particular, for $b=-1$ we apply Theorem \ref{Witt} to this formula to get the equalities $$\gamma(q_1)^2\gamma(-aq_1)^2=\left(\frac{a}{-1}\right) \quad \text{ and } \quad \gamma(aq_1)^2=\left(\frac{a}{-1}\right)\gamma(q_1)^2$$ by Proposition \ref{propgamma1}. Since every non-degenerate quadratic form is isometric to $\sum_{i=1}^m a_ix_i^2$ for suitable $m\in\N$ and $a_i\in F^{\times}$, we have
\begin{equation}\label{formulagamma}
\gamma(f)^2=\prod_{i=1}^m\left(\frac{a_i}{-1}\right)\gamma(q_1)^2=\left(\frac{D(f)}{-1}\right)\gamma(q_1)^{2m}
\end{equation}
where $D(f)$ is the discriminant of $f$. Notice that, since $F$ is non-archimedean, then $-1$ is either a square or a norm in $F(\sqrt{-1})$. Therefore $\gamma(q_4)=\left(\frac{-1}{-1}\right)=1$ and it follows that $\gamma(f)^4=1$ for every non-degenerate quadratic form $f$ over $F$ as announced.\\
This is in fact the best possible result whenever $-1$ is not a square in $F$. Indeed, in this case, there exists at least an element $a\in F^\times$ such that $\left(\frac{a}{-1}\right)=-1$. 
For such an $a$, formula (\ref{formulagamma0}) gives $\gamma(q_1\oplus-aq_1)^2=-1$ 
and then a square root of $-1$ shall be in the image of $\gamma$.
\begin{rmk}
This result shows also that, whenever $-1$ is not a square in $F$ and $\mathrm{char}(R)\neq 2$ (in which case $X^4-1$ is a separable polynomial) then $R$ contains a primitive fourth root of unity. This fact has an elementary explaination: denote $\zeta_p$ an element of order $p$ in $R^\times$ and consider the Gauss sum $\tau=\sum_{i=1}^{p-1} \left(\frac{i}{p}\right) \zeta_p^i\in R$, where in this case $\left(\frac{i}{p}\right)$ is the Legendre symbol. The formula $$\tau^2=\left(\frac{-1}{p}\right)p$$
holds thanks to a classical argument that can be found, for example, in 3.3 of \cite{Lem}. The fact that $-1$ is not a square in $F$ implies that $\left(\frac{-1}{p}\right)=-1$ and that $q=p^f$ with $f$ odd. Since $R$ contains a square root of $q$, then there exists an element $x\in \Rx$ such that $x^2=p$ and $(\tau\cdot\frac{1}{x})^2=-1$: there is a primitive fourth root of unity in $R$.
\end{rmk}


\section{The reduced metaplectic group}

The metaplectic group, associated with $R$ and $\chi$, is an extension of $\Sp$ by $R^\times$ through the short exact sequence (\ref{ses2}). 
We want to understand when this sequence does (or does not) split, looking for positive numbers $n\in \N$ yielding the existence of subgroups $\Mpn{n}$ of $\Mp$ such that $\pi_{|\Mpn{n}}$ is a finite cyclic cover of $\Sp$ with kernel $\mu_n(R)$. 
We show that, for $F$ locally compact non-discrete non-archimedean field, it is possible to construct $\Mpn{2}$. Then we prove that, when $\mathrm{char}(R)\neq2$, $n=1$ does not satisfy the condition above, namely that the sequence (\ref{ses2}) does not split. Finally we show what happens in the simpler case when $\mathrm{char}(R)=2$.

\smallskip
\noindent
For a closer perspective we suppose that, for some $n\in \N$, $\Mpn{n}$ exists and we look at the following commutative diagram with exact rows and columns
$$
\xymatrix{  & 1 \ar[d] & 1 \ar[d] & \\
               1 \ar[r] & \mu_n(R) \ar[r]\ar[d] & \Mpn{n}\ar[r]\ar[d] &\Sp \ar[r]\ar[d]^{id}& 1\\
               1 \ar[r] & R^\times\ar[r]\ar[d]^{\cdot^n} & \Mp\ar[r]^{\pi}\ar[d]^{\psi_n} & \Sp\ar[r] & 1\\
                & R^\times \ar[r]^{id} & R^\times & }
$$
where $\mu_n(R)$ is the group of $n$-th roots of unity in $R$. 
The existence of a homomorphism $\psi_n:\Mp\longrightarrow \Rx$ such that its restriction on $\Rx$ is the $n$-th power map implies the existence of the first line in the diagram. Indeed, if such $\psi_n$ exists, let $\Mpn{n}$ be its kernel; then $\pi$ induces a surjective homomorphism from $\Mpn{n}$ to $\Sp$ whose kernel is $\Mpn{n}\cap \Rx=\mu_n(R)$.\\ Then, as in 43 of \cite{Wei}, the question to address is whether or not there exists $\psi_n: \Mp\to R^\times$ such that $\psi_{n|_{R^\times}}(x)=x^n$ for every $x\in R^\times$.

\begin{lemma}\label{psitilda}
A $R$-character $\psi_n:\Mp\longrightarrow \Rx$ whose restriction on $\Rx$ is the $n$-th power map is completely determined by 
$\widetilde{\psi_n}=\psi_n\circ \r:\Omega(W)\longrightarrow \Rx$ where $\r$ is as in (\ref{definrMp}).
\end{lemma}

\begin{proof}
Let $(\sigma,\s)\in\Mp$. By Proposition \ref{lemma6} we can write $\sigma$ as a product $\sigma=\prod_i \sigma_i$ with $\sigma_i\in\Omega(W)$. We set $(\sigma,\s')=\prod_i \r(\sigma_i)$ where $\r$ is as in (\ref{definrMp}). Then, since $\ker(\pi)=\Rx$, we have that $(\sigma,\s)=c(\sigma,\s')$ for a suitable $c\in\Rx$. This implies that the values of $\psi_n$ at $(\sigma,\s)$ is $\psi_n(c(\sigma,\s'))=c^n\prod_i \widetilde{\psi_n}(\sigma_i)$.
\end{proof}

\noindent
By (\ref{formular}), the morphism $\widetilde{\psi_n}$ of Lemma \ref{psitilda} shall verify the condition 
\begin{equation}\label{relpsi}
\widetilde{\psi_n}(\sigma)\widetilde{\psi_n}(\sigma')=\gamma(f_0)^{n}\widetilde{\psi_n}(\sigma'')
\end{equation}
for every $\sigma=\begin{pmatrix}\alpha&\beta\\ \gamma&\delta\end{pmatrix}$, $\sigma'=\begin{pmatrix}\alpha'&\beta'\\ \gamma'&\delta'\end{pmatrix}$ and $\sigma''=\begin{pmatrix}\alpha''&\beta''\\ \gamma''&\delta''\end{pmatrix}$ in $\Omega(W)$ satisfying $\sigma''=\sigma\sigma'$, where $f_0$ is a non-degenerate quadratic form on $X$ associated to the symmetric isomorphism $-\beta^{-1}\beta''\beta'^{-1}$.
\noindent Conversely we have:

\begin{lemma}\label{omsuf}
If $\widetilde{\psi_n}:\Omega(W)\longrightarrow \Rx$ satisfies (\ref{relpsi}), then there exists a unique $R$-character $\psi_n$ of $\Mp$ such that its restriction to $\Rx$ is the $n$-th power map and $\psi_n\circ \r=\widetilde{\psi_n}$.
\end{lemma} 

\begin{proof}
Let $(\sigma,\s)\in\Mp$. By Proposition \ref{lemma6} we can write $\sigma$ as a product $\sigma=\prod_i \sigma_i$ with $\sigma_i\in\Omega(W)$ and $(\sigma,\s)=c\prod \r(\sigma_i)$ for a suitable $c\in\Rx$. We define $\psi_n(\sigma,\s)=c^n\prod_i \widetilde{\psi_n}(\sigma_i)$. We have to prove that it is well defined. Let $\sigma=\prod_j \sigma_j$ be another presentation of $\sigma$ that differs from $\prod_i\sigma_i$ by a single relation $\sigma\sigma'=\sigma''$; by (\ref{formular}) we obtain 
$$(\sigma,\s)=c\prod_i \r(\sigma_i)=\gamma(f_0)c\prod_j \r(\sigma_j)$$
for a suitable $f_0\in Q^{nd}(X)$ and by (\ref{relpsi}) we have $$\psi_n(\sigma,\s)=c^n\prod_i \widetilde{\psi_n}(\sigma_i)=c^n\gamma(f_0)^n\prod_j \widetilde{\psi_n}(\sigma_j)=(c\,\gamma(f_0))^n\prod_j \widetilde{\psi_n}(\sigma_j).$$
Now, since every presentation $\sigma=\prod_k \sigma_k$ with $\sigma_k\in\Omega(W)$ differs from $\prod_i\sigma_i$ by a finite number of relations $\sigma\sigma'=\sigma''$, the definition $\psi_n(\sigma,\s)=c^n\prod_i \widetilde{\psi_n}(\sigma_i)$ makes sense.
\end{proof}

\noindent After these results the existence of a character $\psi_n$, and then of a subgroup $\Mpn{n}$ of $\Mp$ as above, is equivalent to the existence of $\widetilde{\psi_n}:\Omega(W)\longrightarrow \Rx$ that satisfies (\ref{relpsi}). 

\bigskip
\noindent
First of all we suppose that $-1$ is a square in $F$. 
By Proposition \ref{propgamma1} we have $\gamma(f)^2=1$ for every $f\in Q^{nd}(X)$ and so  $\widetilde{\psi_2}=1$ satisfies (\ref{relpsi}) with $n=2$.

\smallskip

\noindent
We suppose now that $-1$ is not a square in $F$. We fix a basis over the $F$-vector space $X$ and its dual basis over $X^*$. By definition of $\Omega(W)$ we have that the determinant $\det(\beta)$ of $\beta$ with respect to these basis is not zero for every $\sigma=\begin{pmatrix}\alpha&\beta\\ \gamma&\delta\end{pmatrix}\in\Omega(W)$. Moreover, since $f_0$ is associated to the symmetric isomorphism $-\beta^{-1}\beta''\beta'^{-1}$ we have that the discriminant of $f_0$ is $D(f_0)=\det(-\beta)^{-1}\cdot\det(-\beta'')\cdot\det(-\beta')^{-1}$. Hence taking $$\widetilde{\psi_2}(\sigma)=\left(\frac{\det(-\beta)}{-1}\right)\gamma(q_1)^{2m}$$
for every $\sigma=\begin{pmatrix}\alpha&\beta\\ \gamma&\delta\end{pmatrix}\in\Omega(W)$ and using formula (\ref{formulagamma}) we obtain the equality (\ref{relpsi}) with $n=2$.

\noindent We have then proved the 
\begin{teor}\label{rmg} 
There exists a subgroup $\Mpn{2}$ of $\Mp$ that is a cover of $\Sp$ with kernel $\mu_2(R)$. In particular, when $\mathrm{char}(R)\neq 2$, $\Mpn{2}$ is a $2$-cover of $\Sp$.
\end{teor}
\noindent
Now we want to see if this reduction is optimal in the sense that there does not exist any $\Mpn{1}$ fitting into the diagram. If this is the case, then the group $\Mpn{2}$ is the minimal subgroup of $\Mp$ which is a central extension of $\Sp$ and therefore is called \emph{reduced metaplectic group}.

\begin{teor}\label{nonsplitta}
Let $\mathrm{char}(R)\neq 2$. Then there does not exist a character $\psi: Mp(W)\to R^\times$ such that $\psi_{|R^\times}=id$.
\end{teor}
\begin{proof}
Let suppose the existence of such $\psi$. Then there exists a character $\psi': \mathrm{Mp}(F\times F^*)\to R^\times$ such that $\psi'_{|R^\times}=id$. In fact the extension by triviality 
$$\begin{array}{rcc} \iota:\Omega(F\times F^*)& \to&\Omega(W)\\
\begin{pmatrix}a&b\\ c&d\end{pmatrix} & \mapsto & \begin{pmatrix}a&0&b&0\\0&\I_{n-1}&0&\I_{n-1} \\ c&0&d&0\\0&\I_{n-1}&0&\I_{n-1} \end{pmatrix}
 \end{array}$$
is such that $\sigma''=\sigma\sigma'$ yields $\iota(\sigma'')=\iota(\sigma)\iota(\sigma')$. Then $\widetilde \psi':=\widetilde \psi\circ\iota$ satisfies the relation $$\widetilde \psi'(\sigma'')=\gamma(f_0)^{-1}\widetilde \psi'(\sigma)\widetilde \psi'(\sigma')$$
and Lemma \ref{omsuf} implies the existence of $\psi'$. Clearly $\psi'$ takes values $1$ on the group of commutators of $\mathrm{Mp}(F\times F^*)$. By (\ref{relazMp}) we have $$\t\left(\frac{c}{1-a^2}x^2\right)\d\left(a^{-1}\right)\t\left(-\frac{c}{1-a^2}x^2\right)\d\left(a\right)=\t\left(\frac{c}{1-a^2}x^2\right)\t\left(-\frac{ca^2}{1-a^2}x^2\right)=\t\left(cx^2\right)$$
for every $a\notin\{0,1,-1\}$ in $F$ and every $c\in F$. Then for every quadratic form $f$ on $F$, $\t(f)$ is a commutator of $\mathrm{Mp}(F\times F^*)$ and so $\psi'(\t(f))=1$. By Definition \ref{defgamma} we obtain the equality 
$$\d'(\rho^{-1})\t(f)\d'(-\rho^{-1})\t(f)=\gamma(f)\t(-f)\d'(\rho^{-1})$$
in $\mathrm{Mp}(F\times F^*)$ for every $f\in Q^{nd}(F)$ associated to $\rho$ and applying $\psi'$ we obtain $\gamma(f)=\psi'(\d'(\rho^{-1}))$. So, if we denote by $\rho_a$ the symmetric isomorphism associated to $aq_1:x\longrightarrow ax^2$ we obtain
$$\gamma\left(aq_1\right)=\psi'(\d'(\rho_a^{-1}))=\psi'(\d(2a))\psi'(\d'(\rho_1^{-1})).$$
Now, since every quadratic form $f$ over $F$ is of the form $f(x)=\sum_{i=1}^ma_ix_i^2$, we can conclude that $\gamma(f)=\prod_{i=1}^m\psi'(\d(2a_i))\psi'(\d'(\rho_1^{-1}))^m$ depends only on $m$ and on the discriminant. But this implies that $\gamma$ takes the same value on every non-degenerate quadratic form on a $4$-dimensional vector space over $F$ with discriminant equal to $1$. But this contradicts Theorem \ref{gamma-1}.
\end{proof}

\noindent
We shall remark that, if $R$ has characteristic $2$, then necessarily $\gamma(f)=1$ for every quadratic form $f$. Then Theorem \ref{nonsplitta} is clearly false and the sequence (\ref{ses2}) splits yielding the existence of $\Mpn{1}\cong \Sp$.

\bigskip 

\noindent
We conclude by saying that we can restrict the representation of the metaplectic group given by (\ref{Weilrep}) to a representation of the reduced metaplectic group. This is the \textbf{Weil representation} defined over $R$. As pointed out in the introduction, the relevance of having an explicit form for this representation lies in the fact that its understanding has important applications. Considering $R$ in whole generality may help understand more deeply the essential features underlying results like Howe and Shimura correspondences. A more concrete question is the following. Given a morphism of rings $R_1\to R_2$ and fixed two smooth non-trivial characters $\chi_1:F\to R_1$ and $\chi_2:F\to R_2$, it would be interesting to study the relationships between metaplectic groups and the Weil representation respectively over $R_1$ and $R_2$.



\cleardoublepage
\singlespacing

\bibliographystyle{alpha}
\bibliography{groupemetaplectique2}
\end{document}